\documentclass[reqno,12pt]{amsart}
\usepackage{amsmath,amssymb,latexsym,soul,cite,mathrsfs,amsthm}
\usepackage{xcolor,enumitem,graphicx}
\usepackage[colorlinks=true,urlcolor=blue,
citecolor=red,linkcolor=blue,linktocpage,pdfpagelabels,
bookmarksnumbered,bookmarksopen]{hyperref}
\usepackage[english]{babel}
\usepackage[left=2.6cm,right=2.6cm,top=2.9cm,bottom=2.9cm]{geometry}
\usepackage[hyperpageref]{backref}
\newtheorem{theorem}{Theorem}[section]
\newtheorem{corollary}[theorem]{Corollary}
\newtheorem{definition}[theorem]{Definition}
\newtheorem{remark}[theorem]{Remark}
\newtheorem{lemma}[theorem]{Lemma}
\newtheorem{proposition}[theorem]{Proposition}

\numberwithin{equation}{section}

\title[Fractional Trudinger-Moser type inequalities with logarithmic potentials]{Fractional Trudinger-Moser type inequalities with logarithmic convolution potentials}
\author[H. Luo]{Huxiao Luo*}
\author[S. Wang]{Shiying Wang}

\address[H. Luo]{Department of Mathematics, Zhejiang Normal University, Jinhua, Zhejiang, 321004, China}
\email{\href{mailto:luohuxiao@zjnu.edu.cn}{luohuxiao@zjnu.edu.cn}}

\address[S. Wang]{Department of Mathematics, Zhejiang Normal University, Jinhua, Zhejiang, 321004, China}
\email{\href{mailto:wangshiying@zjnu.edu.cn}{wangshiying@zjnu.edu.cn}}

\thanks{*Corresponding author.}


\begin{document}
	\begin{abstract}
		We establish the following fractional Trudinger-Moser type inequality with logarithmic convolution potential
$$ \sup_{u\in W^{\frac{1}{2},2}_0(I),\|u\|_{W_0^{\frac{1}{2},2}}\leq1}\int_{I} \int_{I} \log \frac{1}{|x-y|}  G(u(x))G(u(y))  \, dx \, dy<+\infty,$$
where $G(s)\leq C\frac{e^{\pi s^{2}}}{(1 + |s|)^{\gamma}}~ \forall s\in\mathbb{R}$ with some constant $C>0,\gamma\geq1$,
the domain $I\subset\mathbb{R}$ is a bounded interval. This type of inequality in the entire space $\mathbb{R}$ is also considered.
Moreover, we study the existence of corresponding
extremal functions.

In addition, by the moving plane method, we obtain the radial symmetry and radial decreasing property of positive solutions to the corresponding Euler-Lagrange equation.
\end{abstract}
	\maketitle	
	\begin{center}
		\footnotesize
		\tableofcontents
	\end{center}

\section{Introduction}
 Let $\Omega\subset\mathbb{R}^N$
be a domain with finite measure. The Sobolev embedding theorem states
that $H_0^1(\Omega) \subset L^p(\Omega)$, for $1\leq p\leq \frac{2N}{N-2}$. In the limit case $N=2$, Trudinger \cite{Trudinger}
(see also Yudovich \cite{Yudovich} and Pohozaev \cite{Pohozaev}) have proved that $W^{1,N}_0(\Omega)\hookrightarrow L_{\varphi_N}(\Omega)$, where $L_{\varphi_N}(\Omega)$ is the Orlicz space associated with the Young
function $\varphi_N(t)= e^{\alpha|t|^{\frac{N}{N-1}}} - 1$ for some $\alpha>0$. The best constant of $\alpha$ is given by Moser \cite{Moser}: There exists a positive
constant $C_0$ dependent only on $N$ such that
\begin{equation}\nonumber
\sup\limits_{u\in C^1_c(\Omega),\int_{\Omega}
|\nabla u|^N\leq1}
\int_{\Omega}
e^{\alpha|u|^{\frac{N}{N-1}}}
dx \leq C_0|\Omega|
\end{equation}
holds for all $\alpha \leq \alpha_N = N[\omega_{N-1}]^{\frac{1}{N-1}}$.
Moreover, when $\alpha > \alpha_N$, the above supremum is always infinite.
For more results on Trudinger-Moser type
inequalities, please refer to e.g. \cite{CRuf1, CRuf2, OO1,OO2,OLP1,OLP2,HMS,Ngu,Roy1,Roy2,XZZ}.

In the last decades,
some nonlocal interactions equations with logarithmic kernels have
gained great interest. Such nonlocal problems have arisen from applications in the statistical dynamics of self-gravitating clouds \cite{43},
quantum theory for crystals \cite{20}, and in the description of vortices in turbulent
Euler flows \cite{11}.  The planar Trudinger-Moser type inequalities for energy functionals in the
presence of a logarithmic convolution potential have been established in \cite{18,CWY}.
In \cite{18}, S. Cingolani and T. Weth proved
$$ \sup_{u\in H^{1}(\Omega),\|\nabla u\|_{L^{2}}\leq1}\int_{\Omega} \int_{\Omega} \log \frac{1}{|x-y|}  F(u(x))F(u(y))  \, dx \, dy<+\infty,$$
where $F(s)\leq Ce^{\pi s^{2}}(1 + |s|)^{\beta}~ \forall s\in\mathbb{R}$ with some constant $C>0, \beta\leq-1$,
the domain $\Omega\subset\mathbb{R}^2$ is a ball in $\mathbb{R}^2$. In \cite{CWY}, S. Cingolani, T. Weth and M. Yu sharpened the results in \cite{18} under critical growth assumptions. In \cite{CC}, A. Cannone and S. Cingolani extended the results in \cite{18} to the dimension $N\geq2$.

Inspired by \cite{18}, in this article, we are concerned with one-dimensional Trudinger-Moser type inequalities for energy functionals in the presence of a logarithmic convolution potential. Since the function in this one-dimensional inequality is considered in the critical Sobolev space $W^{\frac{1}{2},2}(\mathbb{R})$, we need the definition of fractional Laplacian and the related fractional Sobolev space. Let
$$L_s(\mathbb{R}^N)=\left\{u\in L^1_{\text{loc}}(\mathbb{R}^N): \int_{\mathbb{R}^N}\frac{|u|}{1+|x|^{N+2s}}dx<\infty\right\}.$$
For any $s\in(0,1)$ and $\varphi\in C_{\text{loc}}^{1,1}(\mathbb{R}^N)\cap L_s(\mathbb{R}^N)$, the fractional Laplacian $(-\Delta)^s$ is given by:
$$(-\Delta)^s\varphi = C_{N,s}\text{P.V.}\int_{\mathbb{R}^N}\frac{\varphi(x)-\varphi(y)}{|x-y|^{N+2s}}dy,$$
where P.V. denotes the Cauchy principal value integral, and $C_{N,s}$ is a normalized constant that depends on $N$ and $s$, precisely given
by
\begin{equation}\label{20250630-e1}
C_{N,s} =\left(\int_{\mathbb{R}^N}\frac{1-\cos{\zeta_1}}{|\zeta|^{N+2s}}d\zeta\right)^{-1}=\frac{4^s s\Gamma(\frac{N}{2}+s)}{\pi^{\frac{N}{2}}\Gamma(1-s)}.
\end{equation}
 If $\varphi\in \mathcal{S}(\mathbb{R}^N)$ (the Schwarz space), $(-\Delta)^s\varphi$  can be equivalently defined \cite{CS,NPV} as:
$$(-\Delta)^s\varphi = \mathscr{F}^{-1}(|\xi|^{2s}\mathscr{F}(\varphi)),$$
where $\mathscr{F}$ and $\mathscr{F}^{-1}$ denote Fourier transform and inverse Fourier transform, respectively.
More generally, the fractional operator $(-\Delta)^s$ can be defined on $L_s(\mathbb{R}^N)$ in the distributed sense:
$$\langle (-\Delta)^su,\varphi\rangle:=\langle (-\Delta)^s\varphi,u\rangle,\quad \forall u\in L_s(\mathbb{R}^N),~\varphi\in \mathcal{S}(\mathbb{R}^N).$$
The fractional Sobolev space is defined as follows
$$W^{s,p}(\Omega):=\left\{u\in L^p(\Omega): [u]_{W^{s,p}(\Omega)}:=\left(\int_{\Omega}\int_{\Omega}\frac{|u(x)-u(y)|^p}{|x-y|^{N+sp}}dxdy\right)^{\frac{1}{p}}<\infty\right\}.$$
\begin{proposition}\label{A.1} (\cite[Proposition 3.6]{NPV})  For $s \in (0, 1)$ we have, $[u]_{W^{s,2}(\mathbb{R}^N)} < \infty$ if and only if $(-\Delta)^{\frac{s}{2}}u\in L^2(\mathbb{R}^N)$, and
in this case
$$[u]_{W^{s,2}(\mathbb{R}^N)}^2 = 2C_{N,s}^{-1}|(-\Delta)^{\frac{s}{2}}u|_{L^2(\mathbb{R}^N)}^2,$$
where $C_{N,s}$ depends by \eqref{20250630-e1}. In particular $H^{s,2}(\mathbb{R}^N) = W^{s,2}(\mathbb{R}^N)$.
\end{proposition}
\begin{remark}
If we replace the entire space $\mathbb{R}^N$ with a domain $\Omega\subset\mathbb{R}^N$, Proposition \ref{A.1} does not hold due to the non-locality of the fractional Laplacian.
\end{remark}
For a more detailed introduction to the fractional Sobolev space, please refer to \cite{NPV} and its references.

In the follows, we will introduce some results related to the fractional Trudinger-Moser type inequalities, see e.g. \cite{CWZ,H,26}.
In \cite{26}, Iula et al. established the following fractional Trudinger-Moser inequality on some interval of line.
\begin{proposition}
(\cite[Theorem 1.1]{26}
    For any $I\subset\subset\mathbb{R}$, when $\alpha\leq\pi$, it holds that
\begin{equation}\label{eq:1.1}\sup_{u\in W_{0}^{\frac{1}{2},2}(I),\|(-\Delta)^{\frac{1}{4}}u\|_{L^2(I)}^{2}\leq1}\int_{I}\left(e^{\alpha u^{2}}-1\right)\mathrm{d}x=C|I|<+\infty,\end{equation}
where $C>0$ is a constant, $|I|$ denotes the measure of $I$.
\end{proposition}

Chen et al. \cite{CWZ} proved that
\begin{proposition}
 Let $I$ be a bounded interval of $\mathbb{R}$ and $\lambda_1(I)$ be the first eigenvalue of $(-\Delta)^{\frac{1}{4}}$
with Dirichlet
boundary. For any $0\leq\alpha<\lambda_1(I)$, there holds
\begin{equation}\ \label{eq:1.2}\sup_{u\in W_{0}^{\frac{1}{2},2}(I),\|(-\Delta)^{\frac{1}{4}}u\|_{L^2(I)}^{2}-\alpha\|u\|_{2}^{2}\leq1}\int_{I}e^{\pi u^{2}}\mathrm{d}x=c(I,\alpha)<+\infty.\end{equation}
\end{proposition}

In this setting, we consider the functional
\begin{equation*}\label{eq:1.3}
\Phi(u) :=  \int_{I} \int_{I} \log \frac{1}{|x-y|} G(u(x)) G(u(y)) \, dx \, dy,\quad \forall u\in W^{\frac{1}{2},2}_0(I),
\end{equation*}
\\
 where the domain $I\subset\mathbb{R}$ is a bounded interval, and\\
(A) $G:\mathbb{R} \rightarrow [0, \infty)$ is even, continuous, and strictly increasing on $[0, \infty)$. Moreover, there exist constants $\alpha, c > 0$ such that
\begin{equation}\label{eq:1.4}
    G(t) \leq c e^{\alpha t^2} \quad \text { for }t \in \mathbb{R}\text{ in }I.
\end{equation}
We aim to focus on the problem of maximizing the functional (\ref{eq:1.3}) among functions in the set
$$
E := \left\{u \in W^{\frac{1}{2},2}_0(I) : \|- \Delta^{\frac{1}{4}} u\|_{L^2} \leq 1\right\}.$$

\begin{definition}\label{7:47}
    \label{de:1.1} Let $\gamma \in \mathbb{R}$, and let $f : \mathbb{R} \to \mathbb{R}$ be an arbitrary function. We say that $f$ has
\\
(i) at most $\gamma$-critical growth if $|f(s)| \leq c \, \frac{e^{\pi s^{2}} }{(1 + |s|)^\gamma}$ for $s \in \mathbb{R}$ with some constant $c > 0$;
\\
(ii) at least $\gamma$-critical growth if there exist $s_{0}, c > 0$ with the property that
$$|f(s)| \geq c \, \frac{e^{\pi s^{2}} }{|s|^{\gamma}} \quad \text{for } |s| \geq s_{0}.$$
The following is our first result of the paper.
\end{definition}

\begin{theorem}
 \label{th:1.2} Suppose that $G$ satisfies $(A)$.\\
 (i) If $G$ has at most $\gamma$-critical growth for some $\gamma \geq1$, then
$$
m(G):=\sup_{u \in E} \Phi(u)<\infty.
$$
\noindent(ii) If $G$ has at most $\gamma$-critical growth for some $\gamma >1$, and  \begin{equation}\label{20250704-e1}
G(s)-G(0) \leq c \, \left(e^{\pi s^{2}} -1\right)\quad \forall s \in \mathbb{R}\end{equation}
with some constant $c > 0$. Then $m(G)$ is attained, and every maximizer for $\Phi$ in $E$ is, up to sign, an even and radially decreasing function in $E$.\\
(iii) If $G$ has at least $\gamma$-critical growth for some $\gamma<1$, then $m(G)=\infty$.
\end{theorem}

It is easy to verify that $G(t)=t^2$ satisfies the conditions in Theorem \ref{th:1.2}-(ii).
  Another example of $G$,
 $$G(s)=
\left\{\begin{aligned}
&e^{\pi s^2}-1& & \text { for } |s|\leq c_\gamma, \\
&\frac{e^{\pi s^2}}{(1+|s|)^\gamma}-\frac{e^{\pi c_\gamma^2}}{(1+|c_\gamma|)^\gamma}+e^{\pi c_\gamma^2}-1 & & \text { for } |s|>c_\gamma,
\end{aligned}\right.
\quad\text{with}~\gamma>1,
$$
where $$c_\gamma=\frac{\sqrt{1+\frac{2\gamma}{\pi}}-1}{2}~\text{is~the~minimum~ point~of}~\frac{e^{\pi s^2}}{(1+|s|)^\gamma}.$$

Next, we consider the logarithmic Trudinger-Moser type inequality in the entire space $\mathbb{R}$. In addition, we need the following additional structural assumption on the function $G$. \\
$\left(A_1\right) G: \mathbb{R} \rightarrow[0, \infty)$ satisfies $(A)$, and $G(t)=O(|t|)$ as $t \rightarrow 0$.

Moreover, for $u\in H^{\frac{1}{2}}(\mathbb{R})$ we define
$$
\begin{aligned}
\Psi(u):=&\int_{\mathbb{R}} \int_{\mathbb{R}} \log \frac{1}{|x-y|} G(u(x)) G(u(y)) d x d y\\
=&\int_{\mathbb{R}} \int_{\mathbb{R}} \log ^{+} \frac{1}{|x-y|} G(u(x)) G(u(y)) d x d y-\int_{\mathbb{R}} \int_{\mathbb{R}} \log ^{+}|x-y| G(u(x)) G(u(y)) d x d y\\
=:&\Psi_+(u)-\Psi_-(u),
\end{aligned}
$$
where $\log^+=\max\{\log,0\}$.
We let $\|u\|=\left(\|(-\Delta)^{\frac{1}{4}}u\|_{2}^{2}+\|u\|_2^2\right) ^{1 / 2}$ denote the $H^\frac{1}{2}$ norm of $u \in H^\frac{1}{2}\left(\mathbb{R}\right)$, and we define the set
$$
E_{\infty}:=\left\{u \in H^\frac{1}{2}\left(\mathbb{R}\right):\|u\| \leq 1\right\}.
$$
\begin{theorem}
    \label{th:1.4} Suppose that $G$ satisfies $\left(A_1\right)$.\\
(i) If $G$ has at most $\gamma$-critical growth for some $\gamma \geq1$, then
$$
m_{\infty}(G):=\sup _{u \in E_{\infty}} \Psi(G(u ))<\infty.
$$
\noindent(ii) If $G$ has at most $\gamma$-critical growth for some $\gamma >1$, and  $G(s) \leq c \, \left(e^{\pi s^{2}}-1\right)$ for $s \in \mathbb{R}$ with some constant $c > 0$. Then $m_{\infty}(G)$ is attained, and every maximizer for $\Psi$ in $E_{\infty}$ is, up to sign and translation, an even and radially decreasing function in $E_{\infty}$.\\
(iii) If $G$ has at least $\gamma$ critical growth for some $\gamma<1$, then
$
m_{\infty}(G)=\infty
$
in the sense that there exists a sequence $\{u_n\}\subset E_\infty$ with $\Psi^\pm(u_n)<\infty$ for every $n\in\mathbb{N}$ and
$\Psi(u_n)\to\infty$ as $n\to\infty$.
\end{theorem}

\textbf{The main innovations of proving the fractional Trudinger-Moser type inequalities in this article:}

(i) When we show the optimality of the growth exponent $\gamma = 1$ in Theorem \ref{th:1.2} and Theorem \ref{th:1.4},
we need to prove that the $H^{\frac{1}{2}}$-norm of Moser sequence $\{w_n\}$ (a segmented function sequence) satisfies the normalization condition and verify  $\Phi(w_n),\Psi(w_n)\to\infty$. {\bf For the fractional Trudinger-Moser type inequality, due to the non-locality  of the fractional  Laplacian, verifying that the $H^{\frac{1}{2}}$-norm of the Moser sequence  satisfies the normalization condition becomes very complex.} We have overcome this difficulty through extensive calculations and very precise asymptotic expansion estimates.
See Proposition \ref{pro:3.3} for details.

(ii) To prove the existence of the extremum function for $ m(G)$ under the assumption that $G$ has at most $\gamma$-critical growth for some $\gamma > 1$, we need to prove the following convergence result:
let $\{u_n\}$ be a sequence in $E$ with $u_n \rightharpoonup 0$ in $H_0^\frac{1}{2}\left(I\right)$, then
\begin{equation*}
    \Phi\left(u_n\right) \rightarrow \Phi(0) \quad \text { as } n \rightarrow \infty,
\end{equation*}
see in Proposition \ref{pro:3.4}. To this end, we need to show
$$\int_{\mathbb{R}}\int_{\mathbb{R}}\log\frac{1}{|x-y|}~1_{I} G(0)1_{I}v_n(y)dxdy=G(0)\int_0^1 v_n(r) (1-r)d r\to0,\quad \text{as}~n\to\infty,$$
where $v_n=G(u_n)-G(0)\geq0$.
Since $v_n$ is not uniformly bounded, we cannot use dominated convergence theorem to obtain the above convergence result.
Although
\begin{equation*}
\begin{aligned}
\lim\limits_{n\to\infty} \int_{\delta}^1 v_n(r) (1-r) d r= 0
\end{aligned}
\end{equation*}
by
\begin{equation}\nonumber
     v_n \rightarrow 0 \quad \text{uniformly~in~} [\delta, 1] ~\text{for ~every~} \delta\in (0,1).
\end{equation}
However, the energy near $0$
$$ \int_0^{\delta} v_n(r) (1-r) d r$$  may not tends towards $0$ as $\delta\to0$, {\bf since the weight function $1-r$  does not tend towards $0$ as $r\to0$.} This is completely different from the two-dimensional situation in \cite{18}. For the two-dimensional situation, this problem reduces to proving the following convergence result
$$\int_{\mathbb{R}^2}\int_{\mathbb{R}^2}\log\frac{1}{|x-y|}~1_{I} G(0)1_{I}v_n(y)dxdy=G(0) \int_0^1 v_n(r) rf(r) d r\to0,$$ where $f(r)=\log\frac{1}{r}\int_0^r\rho d\rho+\int_r^1\rho\log\frac{1}{\rho}d\rho$. Since
the weight function $rf(r)\to0$ as $r\to0$,
$$\left|\int_0^\delta v_nrf(r)dr\right|\leq C\|rf(r)\|_{L^\infty(0,\delta)}\int_{B_1}e^{4\pi u_n^2}dx\to0,~~\text{as}~\delta\to0,$$
thus the convergence result holds.
{\bf This approach fails for one-dimension case.}
To overcome this difficulty, we notice that $\int_{I}\left(e^{\pi u_n^2}-1\right)dx\leq C|I|$ is uniformly bounded by  the fractional Trudinger-Moser inequality \eqref{eq:1.1}. Then by
$|u_n|_p\to0$ and the additional control growth condition \eqref{20250704-e1}, we prove
 $$\int_0^1 [G(u_n(r))-G(0)](1-r)dr\to0$$
by selecting appropriate truncation parameter and using Chebyshev inequality, see Proposition \ref{pro:3.4} for details.

Another goal of this article is to investigate the radially symmetry and radially decreasing of solutions for the corresponding Euler-Lagrange equation.
Under natural additional assumptions, the maximizer $u$ of $\Phi$, $\Psi$ on $E$, $E_{\infty}$ satisfies corresponding Euler-Lagrange equations at least in a weak sense, respectively.
\begin{proposition}\label{pro:1.5} Suppose that $G \in C^1(\mathbb{R})$ satisfies $(A)$, and $g:=G^{\prime}$ satisfies
\begin{equation}\label{eq:5.1}
   g(t) \leq c e^{\alpha t^2} \quad \text { for } t \in \mathbb{R} \text { with constants } c, \alpha>0 .
\end{equation}
Let $u \in E$ be a maximizer of $\left.\Phi\right|_{E}$. Then there exists $\theta>0$ such that $u$ satisfies the Euler-Lagrange equation in weak sense, that is,
$$
\begin{aligned}
    \int_{I} (- \Delta)^{\frac{1}{4}} u (- \Delta)^{\frac{1}{4}} \varphi d x =\theta \int_{I}\left(\log \frac{1}{|\cdot|} *\left(1_{I} G(u)\right)\right) g(u) \varphi d x \quad \text { for all } \varphi \in H_0^\frac{1}{2}\left(I\right).
\end{aligned}
$$
(ii) If $G$ satisfies $\left(A_1\right)$ and $u$ is a maximizer of $\Psi$ on $E_{\infty}$, then there exists $\theta \in \mathbb{R}$ such that $u$ satisfies the associated Euler-Lagrange equation weak sense, that is,
$$
    \int_{\mathbb{R}}\left((- \Delta)^{\frac{1}{4}} u (- \Delta)^{\frac{1}{4}}\varphi+u\varphi\right) d x=\theta \int_{\mathbb{R}}\left(\log \frac{1}{|\cdot|} * G(u)\right) g(u) \varphi d x
$$
for all $\varphi \in H^\frac{1}{2}\left(\mathbb{R}\right)$ with bounded support.
\end{proposition}
\begin{remark}
    The proof of Proposition \ref{pro:1.5} is completely similar to the proof of \cite[Theorem 1.5]{18}, so we will omit it.
\end{remark}

Under the condition in Proposition \ref{pro:1.5}, the maximizer for $\Psi$ in $E_{\infty}$ is the weak solution for
\begin{equation}\nonumber
(- \Delta)^{\frac{1}{2}} u+u  =\theta \left(\log\frac{1}{|x|} \ast G(u)\right) g(u)  \text { in } \mathbb{R}.
\end{equation}
Let $$w= \log\frac{1}{|x|} \ast G(u)=\int_{\mathbb{R}}\log\frac{1}{|x-y|}G(u(y))dy .$$
Then $$(- \Delta)^{\frac{1}{2}}  w  =\pi G(u).$$
Moreover, if $G(u) \in L_{\log }^1\left(\mathbb{R}\right)$ (the space is defined in \eqref{1455}) be a nonnegative radial function, then for $x\in\mathbb{R}\setminus\{0\}$,
$$ \begin{aligned}
w =&\int_{\mathbb{R}}\min\{\log\frac{1}{|x|},\log\frac{1}{|y|}\}G(u(y))dy \\
=&\log\frac{1}{|x|}\int_{|y|<|x|}G(u(y))dy+\int_{|y|\geq|x|}\log\frac{1}{|y|}~~G(u(y))dy,
\end{aligned}$$
see in Lemma \ref{le:2.7}.
Thus, we see
$$w(x) \rightarrow-\infty \quad \text {as}~|x| \rightarrow \infty.
$$

In the final part of the paper, by moving plane method, we prove the radially symmetry and radially decreasing of positive solutions for the corresponding Euler-Lagrange equation.
\begin{theorem}\label{th:1.6}
Suppose that $G\in C^1(\mathbb{R})$ with $g=G^{\prime}$, $G(0)=g(0)=0$, and $g$ is a non-decreasing Lipschitz function on $[0, \infty)$. If $(u, w)$ is a classical solution  of
$$
\left\{\begin{aligned}
(- \Delta)^{\frac{1}{2}} u+u & =\theta w g(u) & & \text { in } \mathbb{R} \\
(- \Delta)^{\frac{1}{2}}  w & =\pi G(u) & & \text { in } \mathbb{R}
\end{aligned}\right.
$$
satisfying
$$
u \in L^{\infty}\left(\mathbb{R}\right) \quad \text { and } \quad w(x) \rightarrow-\infty \quad \text { as }|x| \rightarrow \infty
$$
with $u>0$ in $\mathbb{R}$, then $u$ is even and radially decreasing.
\end{theorem}

\textbf{The main innovations of proving the radially symmetry and radially decreasing of the classical solution $u$ in Theorem \ref{th:1.6}:}

(i) To use the moving plane method of integral type, we need the decay property of the classical solution $u$ at infinity. To this end, by constructing the comparison function $\gamma(x)=\frac{x^2}{1+x^2}+1$ and by calculating $(-\Delta)^{\frac{1}{2}}\gamma(x)$ in the sense of point-wise, we then get $\lim\limits_{|x|\to\infty}u(x)=0$,
see in Lemma \ref{le:6.1}. Here, $(-\Delta)^{\frac{1}{2}}\gamma(x)$ cannot be a distribution, as we will use the maximum principle in the proof of Lemma \ref{le:6.1}.
Moreover, by the Fourier transform, we calculate $(-\Delta)^{\frac{1}{2}}\frac{1}{1+|x|^2}$ in the classic sense and then prove $u(x)=O(\frac{1}{1+|x|^2})$,
see in Lemma \ref{le:6.2}.

Moreover, since the positive solution of the fractional Laplacian equation decays polynomial at infinity, which is different from the exponential decay of the positive solution for the regular Laplacian equation. Therefore, proving that the classical solution belongs to $H^{\frac{1}{2}}(\mathbb{R})$ becomes even more difficult, see Section 5.1.

(ii) To use the moving plane method of integral type, we need prove $$\|w_\lambda^+\|_{L^2(H_\lambda)}\leq C \|u_\lambda^+\|_{L^2(H_\lambda)}$$
by using Hardy-Littlewood-Sobolev inequality, where $$w_\lambda^+(x)=\int_{H_\lambda}\log\frac{|x-y^\lambda|}{|x-y|}K_\lambda(y)u_\lambda(y)dy,$$
see in Lemma \ref{le:6.3}. {\bf Since the dimension of space is $1$, the method in \cite{18} is not applicable to this problem in this article.} To overcome this difficulty, we estimate $w_\lambda^+(x)$ by an improved inequality $\log(1+|x|)\leq\sqrt{|x|}$,
see in Lemma \ref{le:6.3} for details.

\section{Preliminaries}
Let us begin by defining the necessary notation. For a real-valued measurable function \( u \), its \( L^p \) norm is denoted by \( |u|_p \).
We define the space with $\log$ weight function:
\begin{equation}\label{1455}
L_{\log}^1(\mathbb{R}) := \left\{v \in L^1(\mathbb{R}): |v|_*:= \int_{\mathbb{R}} \log(1+|x|)|v|dx < \infty\right\}.
\end{equation}
Let $\log^{+} = \max\{\log, 0\}$. From the fundamental inequality, we have the estimate
\begin{equation}\label{eq:2.3}
   \log^{+}|x-y| \leq \log^{+}(|x|+|y|) \leq \log(1+|x|) + \log(1+|y|) \quad \text{for } x, y \in \mathbb{R},
\end{equation}
then for nonnegative real-valued measurable function $v$ we derive the convolution estimate
\[
\left[\log^{+}|\cdot| * v\right](x) \leq \log(1+|x|)|v|_1 + |v|_* \quad \text{for } v \in L_{\log}^1(\mathbb{R}),\; x \in \mathbb{R}.
\]
This immediately implies the local boundedness property
\begin{equation}\label{eq:2.4}
    \log^{+}|\cdot| * v \in L_{loc}^{\infty}(\mathbb{R}) \quad \text{for } v \in L_{\log}^1(\mathbb{R}).
\end{equation}
Moreover, for nonnegative real-valued measurable functions $v$ and $w$,
we obtain the finiteness condition
\begin{equation}\label{eq:2.5}
\int_{\mathbb{R}}\int_{\mathbb{R}}\log^+|x-y|v(x)w(y)dxdy \leq |v|_1|w|_* + |w|_1|v|_* < \infty \quad \text{for } v, w \in L_{\log}^1(\mathbb{R}).
\end{equation}
\begin{lemma}
    \label{le:2.1}
    Let nonnegative real-valued measurable functions $v$ and $w$ satisfy: $v\not\equiv0$, $w \in L_{loc}^1(\mathbb{R})$ and $$\int_{\mathbb{R}}\int_{\mathbb{R}}\log^+|x-y|v(x)w(y)dxdy < \infty,$$
then $w$ belongs to $L_{\log}^1(\mathbb{R})$.
\end{lemma}
\begin{proof} The idea of this proof comes from \cite[Lemma 2.1]{18}, we provide the detailed process for the convenience of readers.
Given that $v \not\equiv 0$, there exist constants $d > 1$, $c > 0$, and a measurable set $A \subset I_d:=(-d, d)$ with positive measure $|A| > 0$, such that $v \geq c$ holds almost everywhere on $A$. For $x \in A$ and $y \in \mathbb{R} \setminus I_{d+\delta}$, where $\delta$ is a fixed constant in $(0,\infty)$, we have the inequality
\[
|x - y| \geq \frac{\delta|y|}{d+\delta} \geq \delta.
\]
This yields the lower bound estimate
\[
\begin{aligned}
&\int_{\mathbb{R}}\int_{\mathbb{R}}\log^+|x-y|v(x)w(y)dxdy\\
\geq &c \int_A \int_{\mathbb{R}} w(y) \log^{+}|x - y|  dy  dx \\
\geq &c \int_A \int_{\mathbb{R} \setminus I_{d+\delta}} w(y) \log^{+}\left(\frac{|y|}{2}\right)  dy  dx \\
= &c |A| \int_{\mathbb{R} \setminus I_{d+\delta}} w(y) \log^{+}\left(\frac{|y|}{2}\right)  dy\\
\geq& c_1 \int_{\mathbb{R} \setminus I_{d+\delta}} w(y) \log^{+}(1 + |y|)  dy
\end{aligned}
\]
for some constant $c_1 > 0$. The last integral satisfies
\[
c_1 \int_{\mathbb{R} \setminus I_{d+\delta}} w(y) \log^{+}(1 + |y|)  dy \geq c_1 \left( |w|_* - \log(1 + d+\delta) \| w \|_{L^1(I_{d+\delta})} \right).
\]
Since $\int_{\mathbb{R}}\int_{\mathbb{R}}\log^+|x-y|v(x)w(y)dxdy < \infty$ and $w \in L^1_{\text{loc}}(\mathbb{R})$, the expression on the right is finite. This implies $|w|_* < \infty$, and therefore $w \in L_{\log}^1(\mathbb{R})$.
\end{proof}
The following corollary is an immediate consequence of Lemma \ref{le:2.1}.
\begin{corollary}
    \label{cor:2.2} If $v \in L_{l o c}^1\left(\mathbb{R}\right)$ is nonnegative and satisfies $$\int_{\mathbb{R}}\int_{\mathbb{R}}\log^+|x-y|v(x)v(y)dxdy<\infty,$$ then $v \in L_{\log }^1\left(\mathbb{R}\right)$.   \end{corollary}

    We further observe from inequality \eqref{eq:2.5} that for functions $v, w \in L^1(\mathbb{R})$ supported in a bounded interval $I \subset \mathbb{R}$, we have the bounds
\begin{align*}
|v|_* \leq (\log 2) |v|_1,\quad
|w|_* \leq (\log 2) |w|_1.
\end{align*}
Consequently,
\begin{equation}\label{eq:2.6}
\int_{\mathbb{R}}\int_{\mathbb{R}}\log^+|x-y|v(x)w(y)dxdy \leq |v|_1|w|_* + |w|_1|v|_* \leq (2 \log 2) |v|_1 |w|_1 < \infty.
\end{equation}

For any nonnegative measurable function $v$, let $v^*$ denote its Schwarz symmetric rearrangement. Then $v^*$ is radial and non-increasing in the radial coordinate. From \cite[Lemma 2.3]{18}, we have the following rearrangement inequalities.
\begin{lemma}\label{le:2.3}(\cite[Lemma 2.3]{18}) Let $v$ be a nonnegative real-valued measurable function. Then we have:\\
(i) $$\int_{\mathbb{R}}\int_{\mathbb{R}}\log^+\frac{1}{|x-y|}v^*(x)v^*(y)dxdy\geq \int_{\mathbb{R}}\int_{\mathbb{R}}\log^+\frac{1}{|x-y|}v(x)v(y)dxdy;$$
(ii) $$\int_{\mathbb{R}}\int_{\mathbb{R}}\log^+|x-y|v^*(x)v^*(y)dxdy \leq \int_{\mathbb{R}}\int_{\mathbb{R}}\log^+|x-y|v(x)v(y)dxdy;$$
(iii) If $$\int_{\mathbb{R}}\int_{\mathbb{R}}\log^+\frac{1}{|x-y|}v^*(x)v^*(y)dxdy<\infty,$$ then
\begin{equation}\label{eq:2.7}
\int_{\mathbb{R}}\int_{\mathbb{R}}\log\frac{1}{|x-y|}v^*(x)v^*(y)dxdy \geq \int_{\mathbb{R}}\int_{\mathbb{R}}\log\frac{1}{|x-y|}v(x)v(y)dxdy.
\end{equation}
Moreover, if $v \in L^p\left(\mathbb{R}\right)$ for some $p \in[1,2]$, and $$\int_{\mathbb{R}}\int_{\mathbb{R}}\log^+\frac{1}{|x-y|}v(x)v(y)dxdy<\infty,~~\int_{\mathbb{R}}\int_{\mathbb{R}}\log^+|x-y|v(x)v(y)dxdy<\infty,$$  then equality holds in (\ref{eq:2.7}) if and only if $v=v^*\left(\cdot-x_0\right)$ for some $x_0 \in \mathbb{R}$.
\end{lemma}
Here we note that, under the assumptions of (iii), both $$\int_{\mathbb{R}}\int_{\mathbb{R}}\log\frac{1}{|x-y|}v^*(x)v^*(y)dxdy~\text{and}\int_{\mathbb{R}}\int_{\mathbb{R}}\log\frac{1}{|x-y|}v(x)v(y)dxdy$$ are well defined by the inequalities in (i) and (ii).
 The well-known Polya-Szeg{\"o} inequality \cite{15,29,30} states that if $u \in W^{1,p}(\mathbb{R}^N, \mathbb{R}^+)$ then
$$  \int_{\mathbb{R}^N}|\nabla u^*|^2dx \leq
 \int_{\mathbb{R}^N}|\nabla u|^2dx,$$
 where $u^*$ is the Schwarz symmetric rearrangement of $u$. This kind of inequalities still holds in the
 nonlocal case, e.g. for the standard fractional norm, namely for $u \in W^{s,p}(\mathbb{R}^N, \mathbb{R}^+)$
 $$ \int_{\mathbb{R}^N}\int_{\mathbb{R}^N}\frac{|u^{*}(x) - u^{*}(y)|^p}{ |x-y|^{N+ps}}
 dxdy \leq
 \int_{\mathbb{R}^N}\int_{\mathbb{R}^N}\frac{|u(x)-u(y)|^p}{ |x-y|^{N+ps}}
 dxdy,\quad
 \text{for}~ p \geq 1 ~\text{and}~ s \in (0,1), $$
 see e. g. \cite{AL,Bae}.
\begin{lemma} (\cite{AL,Bae})
    \label{le:2.4}For any $u \in H^\frac{1}{2}\left(\mathbb{R}\right)$ we have $u^* \in H^\frac{1}{2}\left(\mathbb{R}\right)$ and
$$
\left|u^*\right|_p=|u|_p~\text { for all } 2 \leq p<\infty, \quad\left|(-\Delta)^{\frac{1}{4}}{u}^*\right|_2 \leq|(-\Delta)^{\frac{1}{4}}{u}|_2.
$$
Consequently, we have $u^* \in E$ if $u \in E$ and $u^* \in E_{\infty}$ if $u \in E_{\infty}$.
\end{lemma}
We now establish nonuniform nonlinear estimates for fractional Sobolev spaces, considering both bounded and unbounded domains.
\begin{lemma}
        \label{le:2.5.1} Let $f: \mathbb{R} \rightarrow \mathbb{R}$ be a continuous function satisfying
 \begin{equation}\label{eq:2.10}
    |f(t)| \leq c e^{\alpha t^2} \quad \text { for } t \in \mathbb{R} \text { with constants } c, \alpha>0.
\end{equation}
(i) If $u \in H_0^\frac{1}{2}\left(I\right)$, then we have $f(u) \in L^s\left(I\right)$ for $1 \leq s<\infty$ and
$$
\log ^{+} \frac{1}{|\cdot|} *\left(1_{I} f(u)\right), \quad \log ^{+}|\cdot| *\left(1_{I} f(u)\right) \quad \in L^{\infty}\left(I\right) .
$$
\noindent(ii) If $u \in H^\frac{1}{2}\left(\mathbb{R}\right)$, then $f(u) \in L_{\text {loc }}^s\left(\mathbb{R}\right)$ for $1 \leq s<\infty$.
In addition, if $|f(t)|=O(|t|)$ as $t \rightarrow 0$, then $f(u) \in L^2\left(\mathbb{R}\right)$ and
$$
\log ^{+} \frac{1}{|\cdot|} * f(u) \quad \in L^s\left(\mathbb{R}\right) \quad \text { for } 2 \leq s \leq \infty.
$$
\end{lemma}
\begin{proof}
(i) Without loss of generality, we assume $$I=(-1,1).$$
Let $u \in H_0^\frac{1}{2}\left(I\right)$ and $v=1_{I} f(u)$.
From \cite{26} (see also \cite{CWZ}), we know that
\begin{equation}\label{eq:2.11}
    \int_{I} e^{\alpha u^2} d x<\infty \quad \text { for any } u \in H_0^\frac{1}{2}\left(I\right) \text { and any } \alpha>0.
\end{equation}
 By (\ref{eq:2.10}) and (\ref{eq:2.11}), we have $v \in L^s\left(\mathbb{R}\right)$ for $s \in[1, \infty)$. On the other hand, we see for $s \in[1, \infty)$
\begin{equation}\label{eq:2.12}
   \int_{\mathbb{R}} \left|\log ^{+} \frac{1}{|x|} \right|^sdx=\int_{0}^1\left|\log x\right|^sdx<\infty.
\end{equation}
Then,  by Young's inequality we see that $$\left|\log ^{+} \frac{1}{|\cdot|} * v \right|_{L^\infty (\mathbb{R})}\leq \left|\log ^{+} \frac{1}{|\cdot|}  \right|_2|v|_2<\infty. $$

Moreover, we deduce that
$$
\left[\log ^{+}|\cdot| * v\right](x)=\int_{I} \log ^{+}|x-y| v(y) d y<(\log 2)|v|_1 \quad \text { for } x \in I
$$
and therefore $\log ^{+}|\cdot| * v \in L^{\infty}\left(I\right)$.

(ii) Let $u \in H^\frac{1}{2}\left(\mathbb{R}\right)$. For given $d>0$ and $s \in[1, \infty)$, by (\ref{eq:2.10})  we may assume without loss of generality that $f(t)=e^{\alpha t^2}$ for $t \in \mathbb{R}$.
Then we have $\left[1_{I_d} f(u)\right]^* \leq 1_{I_d}[f(u)]^*=1_{I_d} f\left(u^*\right)$, and therefore,
$$
\left|1_{I_d}f(u)\right|_s=\left|\left[1_{I_d}f(u)\right]^*\right|_s \leq\left|1_{I_d}f\left(u^*\right)\right|_s,
$$
so it suffices to consider the case where $u=u^*$ in the following.

Let $v:=f(u)$. Since $u$ is locally bounded on $\mathbb{R} \backslash\{0\}$, the same is true for $v$. It thus suffices to prove that $v \in L^s\left(I\right)$. For this we consider the function
$$w(x):=
\left\{\begin{aligned}
&(1+|u|_2^2)^{\frac{1}{2}}\left(u(x)-u(1)\right),&|x|\leq1,\\
&0,&|x|\geq1.\\
\end{aligned}\right.
$$
Since $u$ is radially decreasing function, we have
$$u(x)-u(1)\leq u(x)-u(y)\quad \text{if}~|x|\leq1\leq|y|,$$
and
$$u(y)-u(1)\leq u(y)-u(x)\quad\text{if}~ |y|\leq1\leq|x|.$$
Thus, for $|x|\leq1$,
$$
\begin{aligned}
&\int_{\mathbb{R}} \frac{(w(x)-w(y))^2}{(x-y)^2} d y\\
 =&(1+|u|_2^2)\int_{|y|\leq1} \frac{(u(x)-u(y))^2}{(x-y)^2} d y+(1+|u|_2^2)\int_{|y|\geq1} \frac{\left(u(x)-u\left(1\right)\right)^2}{(x-y)^2} d y \\
 \leq& (1+|u|_2^2)\int_{\mathbb{R}} \frac{(u(x)-u(y))^2}{(x-y)^2} d y.
\end{aligned}.
$$
And for $|x|\geq1$,
$$
\begin{aligned}
\int_{\mathbb{R}} \frac{(w(x)-w(y))^2}{(x-y)^2} d y & =\int_{\mathbb{R}} \frac{(w(y))^2}{(x-y)^2} d y=(1+|u|_2^2)\int_{|y|\leq1} \frac{\left(u(y)-u\left(1\right)\right)^2}{(x-y)^2} d y \\
& \leq (1+|u|_2^2)\int_{\mathbb{R}} \frac{(u(x)-u(y))^2}{(x-y)^2} d y.
\end{aligned}.
$$
Notice that $\int_{\mathbb{R}} \frac{(u(x)-u(y))^2}{(x-y)^2} d y$ converges for a.e. $x \in I$ thanks to Proposition \ref{A.1}, $u\in H^{\frac{1}{2}}(\mathbb{R})$ and Fubini's
theorem. Integrating with $x$ we obtain
$$
\begin{aligned}
\left\|(-\Delta)^{\frac{1}{4}} w\right\|_{L^2(\mathbb{R})}^2 &=
\frac{2}{C_{1,\frac{1}{4}}^2} \int_{\mathbb{R}} \int_{\mathbb{R}} \frac{(w(x)-w(y))^2}{(x-y)^2} d y d x \\
& \leq \frac{2}{C_{1,\frac{1}{4}}^2} (1+|u|_2^2) \int_{\mathbb{R}} \int_{\mathbb{R}} \frac{(u(x)-u(y))^2}{(x-y)^2} d y d x \\
& =(1+|u|_2^2)\left\|(-\Delta)^{\frac{1}{4}} u\right\|_{L^2(\mathbb{R})}^2\\
&<\infty.
\end{aligned}
$$
Thus, by $w(x)=0$ for $|x|\geq1$ and $\|w\|_{H_0^{\frac{1}{2}}(I)}^2\leq \left\|(-\Delta)^{\frac{1}{4}} w\right\|_{L^2(\mathbb{R})}^2$
we obtain $w\in H_0^{\frac{1}{2}}(I)$.

 Since $u$ is even and radially
decreasing, for $x \neq 0$ we have
\begin{equation}\label{20250612-e1}
    u^2(x)\leq\frac{1}{2|x|}\int_{-|x|}^{|x|}u^2(y)dy\leq\frac{|u|_2^2}{2|x|}.
\end{equation}
Then with \eqref{20250612-e1} and the estimate $2a \leq a^2 + 1$, we obtain for $\forall x\in I$,
\begin{equation}\label{eq:2.5.1}
 \begin{aligned}
  u^2(x)&=(u-u(1))^2+2(u-u(1))u(1)+u^2(1)\\
  &\leq(u-u(1))^2+(u-u(1))^2u^2(1)+1+u^2(1)\\
 & \leq(u-u(1))^2+(u-u(1))^2\frac{|u|_2^2}{2}+1+u^2(1)\\
  &\leq w^2(x)+1+u^2(1).
 \end{aligned}
\end{equation}
Then by (\ref{eq:2.11}) and \eqref{eq:2.5.1}, we have
$$
\int_{I}|v|^s d x=\int_{I} e^{ s \alpha u^2} d x \leq e^{(1+u^2(1))s \alpha} \int_{I} e^{s \alpha w^2} d x<\infty,
$$
that implies $v \in L_{l o c}^s\left(\mathbb{R}\right)$.

Next we assume that $|f(t)|=O(|t|)$ as $t \rightarrow 0$. To prove that $v=f(u) \in L^2\left(\mathbb{R}\right)$, it suffices, by \eqref{eq:2.10}, to consider the case where $f(t)=e^{\alpha t^2}-1$.
It then follows that $|f(u)|_2=\left|[f(u)]^*\right|_2=$ $\left|f\left(u^*\right)\right|_2$, so we may assume again that $u=u^*$.
Then $u$ is bounded on $\mathbb{R} \backslash I$, and therefore, it follows $|v| \leq C|u|$ on $\mathbb{R} \backslash I$, so $v 1_{\mathbb{R} \backslash I} \in L^2\left(\mathbb{R}\right)$ since $u \in L^2\left(\mathbb{R}\right)$. Since we already proved that $v 1_{I} \in L^2\left(\mathbb{R}\right)$, we infer that $v \in L^2\left(\mathbb{R}\right)$. Then by (\ref{eq:2.12}) and Young's inequality, it follows that
$$\left|\log ^{+} \frac{1}{|\cdot|} * v \right|_s\leq \left|\log ^{+} \frac{1}{|\cdot|} \right|_{\tilde{s}} |v|_2<\infty,\quad \forall \tilde{s}\in[1,\infty)$$ where
$$\frac{1}{s}+1=\frac{1}{\tilde{s}}+\frac{1}{2}$$
and $2 \leq s \leq \infty$.
\end{proof}

Now we define functionals $\Phi_{\pm} : E \to [0, \infty]$ as the positive and negative part of the kernel $\log \frac{1}{|\cdot|}$ respectively by
\begin{equation}\nonumber
\Phi_{+}(u) = \int_{I} \int_{I} \log^{+} \frac{1}{|x-y|} G(u(x)) G(u(y)) \, dx dy,
\label{eq:1.5}
\end{equation}
\begin{equation}\nonumber
\Phi_{-}(u) = \int_{I} \int_{I} \log^{+} |x-y| G(u(x)) G(u(y)) \, dx dy,
\label{eq:1.6}
\end{equation}
where $\log^{+} = \max\{\log, 0\}$. Then
\begin{equation}\nonumber
\Phi(u)= \int_{I} \int_{I} \log \frac{1}{|x-y|} G(u(x)) G(u(y)) \, dx dy = \Phi_{+}(u) - \Phi_{-}(u),
\label{eq:1.7}
\end{equation}
for every $u \in W^{\frac{1}{2},2}_0(I)$.
\begin{corollary}
    \label{cor:2.6}(i) If $G$ satisfies assumption $(A)$ and $u \in H_0^\frac{1}{2}\left(I\right)$, then $\Phi_{ \pm}(u)<\infty$.\\
(ii) If $G$ satisfies assumption $\left(A_1\right)$ and $u \in H^\frac{1}{2}\left(\mathbb{R}\right)$, then $\Psi_{+}(u)<\infty$.
\end{corollary}
\begin{proof}
    This is a consequence of Lemma \ref{le:2.5.1}, applied to $f=G$.
\end{proof}

\begin{lemma}
\label{le:2.7}
    Let $v \in L_{\log }^1\left(\mathbb{R}\right)$ be a nonnegative radial function. Then the convolution $[(\log |\cdot|) *$ $v](x) \in \mathbb{R}$ is well defined for $x \in \mathbb{R} \backslash\{0\}$ and given by
\begin{equation}\label{eq:2.13}[(\log |\cdot|) * v](x)=\log |x| \int_{I_{|x|}} v(y) d y+\int_{\mathbb{R} \backslash I_{|x|}} \log |y| v(y) d y,
\end{equation}
where $$I_{|x|}:=\{y\in\mathbb{R}: |y|<|x|\}.$$
\end{lemma}
\begin{proof} The idea of this proof comes from \cite[Lemma 2.7]{18}, we provide the detailed process for the convenience of readers.
     Let $x \in \mathbb{R} \backslash\{0\}$. First we consider the nonnegative radial functions $v \in L^{\infty}\left(\mathbb{R}\right)$ with bounded support, we can easily verify that both the LHS and the RHS of (\ref{eq:2.13}) are well defined.
     Therefore, by Newton's theorem \cite[Theorem 9.7]{Lieb}, $$(\log |\cdot|) * v =\int_{\mathbb{R}}\min\{\log|x|,\log|y|\}v(y)dy=\log|x|\int_{|y|<|x|}v(y)dy+\int_{|y|\geq|x|}\log|y|~v(y)dy,$$
     i.e., (\ref{eq:2.13}) holds for any $v \in L^{\infty}\left(\mathbb{R}\right)$ with bounded support.

Next we let $v \in L_{\log }^1\left(\mathbb{R}\right)$ be an arbitrary nonnegative radial function. From the inequality \eqref{eq:2.3}, we derive the bound
$$
\left[\left(\log ^{+}|\cdot| * v\right](x) \leq \log (1+|x|)|v|_1+|v|_*<\infty .\right.
$$
Consider an increasing sequence $\{v_n\}$ of nonnegative radial functions in $L^\infty(\mathbb{R})$ with bounded support such that $v_n \to v$ point-wise. Then
$$
 \left[\left(\log ^{+}|\cdot|\right) * v_n\right](x) \leq\left[\left(\log ^{+}|\cdot|\right)  * v\right](x) \quad \text { for all } n\in\mathbb{N},~x\in\mathbb{R},$$  and then by  monotone convergence theorem,
$$ \lim _{n \rightarrow \infty}\left[\left(\log ^{+}|\cdot|\right) * v_n\right](x)=\left[\left(\log ^{+}|\cdot|\right) * v\right](x) .
$$
 Similarly, applying the monotone convergence theorem to the convolution
\begin{align*}
    &{\left[\left(\log ^{+} \frac{1}{|\cdot|}\right) * v\right](x) } \\
     =&\lim _{n \rightarrow \infty}\left[\left(\log ^{+} \frac{1}{|\cdot|}\right) * v_n\right](x) \\
    =&\lim _{n \rightarrow \infty}\left(\left[\left(\log ^{+}|\cdot|\right) * v_n\right](x)-\left[(\log |\cdot|) * v_n\right](x)\right) \\
=&\left[\left(\log ^{+}|\cdot|\right) * v\right](x)-\lim _{n \rightarrow \infty}\left(\log |x| \int_{I_{|x|}} w_n(y) d y+\int_{\mathbb{R} \backslash I_{|x|}} \log |y| w_n(y) d y\right). \\
=&\left[\left(\log ^{+}|\cdot| \right)* v\right](x)-\left(\log |x| \int_{I_{|x|}} v(y) d y+\int_{\mathbb{R} \backslash I_{|x|}} \log |y| v(y) d y\right).
\end{align*}
Using the decomposition $\log|t| = \log^{+}|t| - \log^{+}(1/|t|)$, we conclude
$$
\begin{aligned}
{[(\log |\cdot|) * v](x) } & =\left[\left(\log ^{+}|\cdot|\right) * v\right](x)-\left[\left(\log ^{+} \frac{1}{|\cdot|}\right) * v\right](x). \\
& =\log |x| \int_{I_{|x|}} v(y) d y+\int_{\mathbb{R} \backslash I_{|x|}} \log |y| v(y) d y,
\end{aligned}
$$
which completes the proof.
\end{proof}

\begin{corollary}\label{cor:2.8}
 Let $v, w \in L_{l o c}^1\left(\mathbb{R}\right)$ be nonnegative even functions.\\
(i) If $v \not \equiv 0, w \not \equiv 0$ and $$\int_{\mathbb{R}}\int_{\mathbb{R}}\log^+\frac{1}{|x-y|}~v(x)w(y)dxdy<\infty,~~\int_{\mathbb{R}}\int_{\mathbb{R}}\log^+|x-y|~v(x)w(y)dxdy<\infty,$$ then $v, w \in L_{\log }^1\left(\mathbb{R}\right)$ and
\begin{equation}\label{eq:2.14}
       \begin{aligned}
&\int_{\mathbb{R}}\int_{\mathbb{R}}\log\frac{1}{|x-y|}~v(x)w(y)dxdy\\
    =&\int_0^{\infty}  v(r)\left(\log \frac{1}{r} \int_0^r w(\rho) d \rho+\int_r^{\infty} \left(\log \frac{1}{\rho}\right) w(\rho) d \rho\right) d r \\
=&\int_0^{\infty}w(r)\left(\log \frac{1}{r} \int_0^r  v(\rho) d \rho+\int_r^{\infty}\left(\log \frac{1}{\rho}\right) v(\rho) d \rho\right) d r.
\end{aligned}
\end{equation}
(ii) If $$\int_{\mathbb{R}}\int_{\mathbb{R}}\log^+\frac{1}{|x-y|}~v(x)v(y)dxdy<\infty,~~\int_{\mathbb{R}}\int_{\mathbb{R}}\log^+|x-y|~v(x)w(y)dxdy<\infty,$$ then $v \in L_{\log }^1\left(\mathbb{R}\right)$ and
\begin{equation}\label{2.15}
    \int_{\mathbb{R}}\int_{\mathbb{R}}\log\frac{1}{|x-y|}~v(x)v(y)dxdy=2 \int_0^{\infty}  v(r) \log \frac{1}{r} \int_0^r  v(\rho) d \rho d r.
\end{equation}

\end{corollary}
\begin{proof}
    (i) The representation \eqref{eq:2.14} follows from Lemma \ref{le:2.1}, Lemma \ref{le:2.7}, and Fubini's theorem.\\
(ii) If $v \equiv 0$, the result holds trivially. For $v \not\equiv 0$, apply part (i) with $w = v$ to get:
\begin{align*}
&\int_{\mathbb{R}}\int_{\mathbb{R}}\log\frac{1}{|x-y|} v(x)v(y) dx dy \\
= &\int_0^{\infty} v(r) \log\frac{1}{r} \left( \int_0^r v(\rho) d\rho \right) dr + \int_0^{\infty} v(r) \left( \int_r^{\infty} \log\frac{1}{\rho} v(\rho) d\rho \right) dr \\
= & 2\int_0^{\infty} v(r) \log\frac{1}{r} \left( \int_0^r v(\rho) d\rho \right) dr.
\end{align*}
\end{proof}

\begin{remark}\label{re:2.9}
    (i) Assume $G$ satisfies condition $(A)$. By Corollary \ref{cor:2.6}-(i) and Corollary \ref{cor:2.8}-(ii), for any $u \in H_0^{\frac{1}{2}}(I) \setminus \{0\}$ we have
$$
\begin{aligned}
\Phi(u)= & 2\int_0^1  G(u(r)) \log \frac{1}{r} \int_0^r G(u(\rho)) d \rho d r\\>&2\int_0^1 G(0) \log \frac{1}{r} \int_0^r G(0) d \rho d r=\Phi(0).
\end{aligned}
$$
This implies $m(G) \in (\Phi(0), \infty]$, so the minimum $m(G)$ cannot be attained at $u = 0$.

(ii) Assume $G$ satisfies condition $(A_1)$, we have
$$
\Psi(u)=\Phi(u)>\Phi(0)=\Psi(0)=0 \quad \text { for every function } u \in H_0^\frac{1}{2}\left(I\right) \subset H^\frac{1}{2}\left(\mathbb{R}\right) \text { with } u \not \equiv 0 .
$$
 Consequently, $m_{\infty}(G) \in (0, \infty]$, and thus $m_{\infty}(G)$ cannot be attained at $u = 0$.
\end{remark}

We now establish a key estimate for use in subsequent sections.
\begin{lemma}\label{le:2.10}
Given nonnegative parameters $\gamma_1, \gamma_2$ with $\gamma_1 + \gamma_2 \geq 2$,
the functional $\Phi_{\gamma_1, \gamma_2}: E \times E \to [0, \infty)$ defined by

\begin{equation}\label{eq:2.16}
    \Phi_{\gamma_1, \gamma_2}\left(u_1, u_2\right):=\int_0^1\frac{e^{\pi u_1^2(r)}}{\left(1+\left|u_1(r)\right|\right)^{\gamma_1}} \log \frac{1}{r} \int_0^r \frac{e^{\pi u_1^2(\rho)}}{ \left(1+\left|u_2(\rho)\right|\right)^{\gamma_2}} d \rho d r
\end{equation}
is bounded.
\end{lemma}
\begin{proof}
    Fix $\varepsilon_0 \in (0, \frac{1}{2\pi})$, for arbitrary $u_1, u_2 \in E$, we define the sets
$$
A_i^{+}:=\left\{r \in(0,1]: u_i(r) \geq \sqrt{\varepsilon_0(-\log r)}\right\},$$  $$A_i^{-}:=\left\{r \in(0,1]: u_i(r)<\sqrt{\varepsilon_0(-\log r)}\right\}
$$
for $i=1,2$. Introducing $v_i:=\frac{ e^{ \pi u_i^2}}{\left(1+\left|u_i\right|\right)^{\gamma_i}}$ we see
\begin{equation}\label{eq:2.18}
    v_i(r) \leq \frac{e^{\pi u_i^2(r)}}{(1+\sqrt{\varepsilon_0(-\log r)})^{\gamma_i}} \quad \text { for } r \in A_i^{+},
\end{equation}
and
\begin{equation}\label{eq:2.17}
    v_i(r) \leq e^{ \pi u_i^2(r)} \leq r^{- \pi \varepsilon_0} \quad \text { for } r \in A_i^{-}.
\end{equation}
Decompose $\Phi_{\gamma_1, \gamma_2}$ into four terms:
$$
\begin{aligned}
& \Phi_{\gamma_1, \gamma_2}\left(u_1, u_2\right)\\
=&\int_{A_1^{+}} v_1(r) \log \frac{1}{r} \int_{A_2^{+} \cap[0, r]} v_2(\rho) d \rho d r+\int_{A_1^{+}} v_1(r) \log \frac{1}{r} \int_{A_2^{-} \cap[0, r]} v_2(\rho) d \rho d r\\&+\int_{A_1^{-}} v_1(r) \log \frac{1}{r} \int_{A_2^{-} \cap[0, r]} v_2(\rho) d \rho d r
 +\int_{A_1^{-}} v_1(r) \log \frac{1}{r} \int_{A_2^{+} \cap[0, r]} v_2(\rho) d \rho d r.
\end{aligned}
$$
  By \eqref{eq:2.17}, we have
$$
\begin{aligned}
\int_{A_1^{-}} v_1(r) \log \frac{1}{r} \int_{A_2^{-} \cap[0, r]} v_2(\rho) d \rho d r  \leq &\int_{A_1^{-}} r^{- \pi \varepsilon_0} \log \frac{1}{r} \int_{A_2^{-} \cap[0, r]} \rho^{-\pi \varepsilon_0} d \rho d r \\
 \leq& \frac{1}{1-\pi \varepsilon_0} \int_0^1 r^{1-2 \pi \varepsilon_0} \log \frac{1}{r} d r\\
 <&C_1,
\end{aligned}
$$
and by $\gamma_1 + \gamma_2 \geq 2$ and the fractional Trudinger-Moser inequality \eqref{eq:1.2},
$$
\begin{aligned}
& \int_{A_1^{+}} v_1(r) \log \frac{1}{r} \int_{A_2^{+} \cap[0, r]} v_2(\rho) d \rho d r \\
 \leq &\int_{A_1^{+}} \frac{e^{\pi u_1^2(r)}}{(1+\sqrt{\varepsilon_0(-\log r)})^{\gamma_1}} \log \frac{1}{r} \int_{A_2^{+} \cap[0, r]} \frac{e^{\pi u_2^2(\rho)}}{(1+\sqrt{\varepsilon_0(-\log \rho)})^{\gamma_2}} d \rho d r\\
 \leq &\int_0^1(-\log r) \frac{e^{\pi u_1^2(r)}}{(1+\sqrt{\varepsilon_0(-\log r)})^{\gamma_1+\gamma_2}} \int_0^r \rho e^{\pi u_2^2(\rho)} d \rho d r \\
 \leq &\frac{1}{\varepsilon_0}\left(\int_0^1 e^{\pi u_1^2(r)} d r\right)\left(\int_0^1 \rho e^{\pi u_2^2(\rho)} d \right)\\ <&C_2,
\end{aligned}
$$
where $C_1$ and $C_2$ are positive constants independent of $u_1, u_2 \in E$.
 We also have
$$
\begin{aligned}
& \int_{A_1^{+}} v_1(r) \log \frac{1}{r} \int_{A_2^{-} \cap[0, r]} v_2(\rho) d \rho d r \\ \leq &\int_0^1 \frac{ e^{\pi u_1^2(r)} }{(1+\sqrt{\varepsilon_0(-\log r)})^{\gamma_1}}\log \frac{1}{r} \int_0^r \rho^{-\pi \varepsilon_0} d \rho d r \\ \leq &\frac{1}{1-\pi \varepsilon_0} \int_0^1 r^{1-\pi \varepsilon_0}\left(\log \frac{1}{r}\right) e^{\pi u_1^2(r)} d r \\ \leq &C_3 \int_0^1 e^{\pi u_1^2(r)} d r \leq C_4,
\end{aligned}
$$
and
$$
\begin{aligned}
& \int_{A_1^{-}} v_1(r) \log \frac{1}{r} \int_{A_2^{+} \cap[0, r]} v_2(\rho) d \rho d r \\ \leq &\int_0^1 r^{-\pi \varepsilon_0} \log \frac{1}{r} \int_0^r e^{\pi u_2^2(\rho)} d \rho d r \\
 \leq &\int_0^1 r^{-\pi \varepsilon_0} \log \frac{1}{r} d r \int_0^1 e^{\pi u_2^2(\rho)} d \rho \\ \leq &C_5 \int_0^1 r^{-\pi \varepsilon_0} \log \frac{1}{r} d r \leq C_6,
\end{aligned}
$$
again via the fractional Trudinger-Moser inequality (\ref{eq:1.2}). Here $C_3,\cdot\cdot\cdot,C_6$ are also positive constants independent of $u_1, u_2 \in E$.
This establishes boundedness, thus proof is finished.
\end{proof}

\section{Maximization problem on the interval  }
This section concludes the proof of Theorem \ref{th:1.2}.  We work within the space $E := \{u \in W^{\frac{1}{2},2}_0(I) : \|- \Delta^{\frac{1}{4}} u\|_{L^2(I)} \leq 1\}$.
\begin{lemma}\label{le:3.1}
    Let $G$ satisfies hypothesis  $(A)$ and has at most $0$-critical growth. Then the functional $\Phi_{-}$ is uniformly bounded on $E$.
\end{lemma}
\begin{proof}
     Consider $u \in E$. Utilizing definition and \eqref{eq:2.6} we have
$$
\Phi_{-}(u)=\int_{\mathbb{R}}\int_{\mathbb{R}}\log^+|x-y|~1_{I}G(u(x))1_{I}G(u(y))dxdy \leq(2 \log 2)\left|1_{I} G(u)\right|_1^2,
$$
the $L^1$-norm admits the bound
\begin{equation*}
\left\| \mathbf{1}_I G(u) \right\|_{L^1(\mathbb{R})}
= \int_I |G(u(x))|  dx
\leq c_1 \int_I e^{\pi u(x)^2}  dx.
\end{equation*}
By the fractional Trudinger-Moser inequality \eqref{eq:1.2}, there exists $c(I) > 0$ such that
\begin{equation*}
\sup_{u \in E} \int_I e^{\pi u(x)^2}  dx \leq c(I).
\end{equation*}
Consequently,
\begin{equation*}
\Phi_{-}(u) \leq (2 \log 2) (c_1 c(I))^2,
\end{equation*}
where the right-hand side is independent of $u \in E$. This establishes the uniform boundedness of $\Phi_{-}$ on $E$.
\end{proof}

\begin{proposition}\label{pro:3.2}
   Assume $G$ satisfies condition $(A)$ and has at most $\gamma$-critical growth for some $\gamma \geq 1$.
    Then we have
$$
m(G) \leq m^{+}(G)<\infty, \quad \text { where } ~ m(G):=\sup_{E} \Phi ~ \text { and } ~ m^{+}(G):=\sup _{E} \Phi^{+} .
$$
\end{proposition}
\begin{proof}
The growth assumption implies
\begin{equation}\label{eq:3.2}
    G(s) \leq c_1 \frac{e^{ \pi s^2}}{1+|s|} \quad \text { for every } s \in \mathbb{R}
\end{equation}
for some $c_1 > 0$. For arbitrary $u \in E$, define $v := \mathbf{1}_I G(u)$.
Corollary~\ref{cor:2.6} yields $$\int_{\mathbb{R}}\int_{\mathbb{R}}\log^+\frac{1}{|x-y|}~v(x)v(y)dxdy<\infty,~~\int_{\mathbb{R}}\int_{\mathbb{R}}\log^+|x-y|~v(x)v(y)dxdy<\infty.$$ Applying Corollary~\ref{cor:2.8}-(ii) and the growth bound~\eqref{eq:3.2}, we obtain
\begin{align*}
\Phi(u)
&= \iint_{\mathbb{R}^2} \log \frac{1}{|x-y|} v(x)v(y)  dxdy \\
&= 2 \int_0^1 v(r) \log \frac{1}{r} \left( \int_0^r v(\rho)  d\rho \right) dr \\
&\leq c_2 := c_1^2 \sup_{u \in E} \Phi_{1,1}(u,u) < \infty.
\end{align*}
The finiteness of $c_2$ follows from Lemma~\ref{le:2.10} with $\gamma_1 = \gamma_2 = 1$. Combining with Lemma~\ref{le:3.1}, the uniform boundedness of $\Phi_{-}$, we derive
\begin{equation}\label{3.3}
    \Phi_{+}(u) = \Phi(u) + \Phi_{-}(u) \leq c_3 := 2c_2 + \sup_{E} \Phi_{-} < \infty
\end{equation}
for all $u \in E$. Consequently, $m^{+}(G) \leq c_3 < \infty$, which implies $m(G) \leq m^{+}(G) < \infty$.
\end{proof}

The following proposition shows the critical growth exponent $\gamma = 1$ is optimal for ensuring $m(G) < \infty$.
\begin{proposition}\label{pro:3.3}
    Suppose that $G$ satisfies $(A)$ and has at least $\gamma$-critical growth for some $\gamma<1$. Then there exists a sequence of functions $w_n \in E \cap L^{\infty}\left(I\right)$ with
    $$\lim_{n \to \infty} \Phi(w_n) = \infty.$$
\end{proposition}

\begin{proof}
     The assumption implies the existence of a constant $c_1>0$ with
\begin{equation}\label{eq:3.4}
    G(s) \geq \frac{ e^{ \pi s^2}}{c_1s^{\gamma}} \quad \text { for } s \geq c_1.
\end{equation}
For $n \in \mathbb{N}$, $n \geq 2$, we now define $w_n $ as
$$w_n(x)=
\left\{\begin{aligned}
&0,&|x|\geq1,\\
&A_n\log (\frac{1}{|x|}), &\frac{1}{n} \leq|x| \leq 1,\\
&A_n\log n, &\quad |x| \leq \frac{1}{n},
\end{aligned}\right.
$$
where
$$
A_n=\frac{1}{\sqrt{ \pi}} \frac{1}{(\log n)^{1 / 2}}\left(1-\frac{1}{\log n}\right)^{1 / 2}.
$$ Obviously, $w_n \in L^{\infty}\left(I\right)$.

Next, we show $w_n \in H_0^\frac{1}{2}\left(I\right)$ and $[w_n]_{H^{\frac{1}{2}}(\mathbb{R})}\leq1$ for $n$ sufficiently large.
To this end, we need calculate the following integral
$$\begin{aligned}
&\int_{\mathbb{R}}\int_{\mathbb{R}} \frac{(w_n(x)-w_n(y))^2}{(x-y)^2} d ydx\\
=&\int_{|x|\geq1}\int_{\mathbb{R}} \frac{(w_n(x)-w_n(y))^2}{(x-y)^2} d ydx+\int_{\frac{1}{n}\leq|x|\leq1}\int_{\mathbb{R}} \frac{(w_n(x)-w_n(y))^2}{(x-y)^2} d ydx\\
&+\int_{|x|\leq1}\int_{\mathbb{R}} \frac{(w_n(x)-w_n(y))^2}{(x-y)^2} d ydx\\
:=&I_1+I_2+I_3.
\end{aligned}$$
For $I_1$, we have
$$
\begin{aligned}
I_1=&\int_{|x|\geq1}\int_{\mathbb{R}} \frac{(w_n(x)-w_n(y))^2}{(x-y)^2} d ydx\\
=&
\int_{|x|\geq1}\int_{\frac{1}{n}\leq|y|\leq1} \frac{A_n^2\left(\log (\frac{1}{|y|})\right)^2}{(x-y)^2} d ydx+
\int_{|x|\geq1}\int_{|y|\leq\frac{1}{n}} \frac{A_n^2\left(\log n)\right)^2}{(x-y)^2} d ydx\\
:=&A_n^2(I_{12}+I_{13}).
\end{aligned}
$$
For $I_2$,
$$
\begin{aligned}
I_2=&\int_{\frac{1}{n}\leq|x| \leq1}\int_{\mathbb{R}} \frac{(w_n(x)-w_n(y))^2}{(x-y)^2} d ydx\\
=& \int_{\frac{1}{n}\leq|x| \leq1}\int_{|y|\geq1} \frac{(A_n\log (\frac{1}{|x|}))^2}{(x-y)^2} d ydx+
\int_{\frac{1}{n}\leq|x| \leq1}\int_{\frac{1}{n}\leq|y| \leq1} \frac{A_n^2\left(\log (\frac{1}{|x|})-\log (\frac{1}{|y|})\right)^2}{(x-y)^2} d ydx\\
&+
\int_{\frac{1}{n}\leq|x| \leq1}\int_{|y|\leq\frac{1}{n}} \frac{A_n^2\left(\log (\frac{1}{|x|})-\log n\right)^2}{(x-y)^2} d ydx\\
:=&A_n^2(I_{21}+I_{22}+I_{23}).
\end{aligned}
$$
For $I_3$, we estimate
$$
\begin{aligned}
I_3=&\int_{|x| \leq\frac{1}{n}}\int_{\mathbb{R}} \frac{(w_n(x)-w_n(y))^2}{(x-y)^2} d ydx\\
 =&\int_{|x| \leq\frac{1}{n}}\int_{|y|\geq1} \frac{A_n^2(\log n)^2}{(x-y)^2} d ydx+\int_{|x| \leq\frac{1}{n}}
\int_{\frac{1}{n}\leq|y| \leq1} \frac{A_n^2\left(\log n-\log (\frac{1}{|y|})\right)^2}{(x-y)^2} d ydx\\
:=&A_n^2(I_{31}+I_{32}).
\end{aligned}
$$
(i) By calculate,
$$\begin{aligned}
I_{12}=I_{21}&=2\int_1^{+\infty}\int_{\frac{1}{n}}^{1}\frac{(\log \frac{1}{y})^2}{(x-y)^2}dydx+2\int_1^{+\infty}\int_{\frac{1}{n}}^{1}\frac{(\log \frac{1}{y})^2}{(x+y)^2}dydx\\
&=2\int_{\frac{1}{n}}^1\frac{(\log y)^2}{1-y}dy+2\int_{\frac{1}{n}}^1\frac{(\log y)^2}{1+y}dy\\
&=2\int_0^{\log n}\frac{t^2}{e^t-1}dt+2\int_0^{\log n}\frac{t^2}{e^t+1}dt\\
&\leq 2\int_0^{+\infty}\frac{t^2}{e^t-1}dt+2\int_0^{+\infty}\frac{t^2}{e^t+1}dt\\
&\leq C.
\end{aligned}$$
(ii) For $I_{13}$ and $I_{31}$,
$$
\begin{aligned}I_{13}=I_{31}
=&(\log n)^2\cdot 2\int_1^{+\infty}\int_{-\frac{1}{n}}^{\frac{1}{n}}\frac{1}{(x-y)^2}dydx\\
=&(\log n)^2\cdot 2\int_1^{+\infty}\frac{\frac{2}{n}}{x^2-\frac{1}{n^2}}dx\\
=&(\log n)^2\cdot2\log\left(1+\frac{2}{n-1}\right).\end{aligned}$$
(iii) For $I_{22}$,
$$\begin{aligned}I_{22}
=2\int_{\frac{1}{n}}^1\int_{\frac{1}{n}}^1\frac{(\log y-\log x)^2}{(x-y)^2}dydx+2\int_{\frac{1}{n}}^{1}\int_{\frac{1}{n}}^1\frac{(\log y-\log x)^2}{(x+y)^2}dydx.
\end{aligned}$$
 Let $u=\log\frac{1}{x},~v=\log\frac{1}{y}$, and $ t=u-v$, $f(t)=\frac{t^2}{e^t+e^{-t}-2}$, $h(t)=\frac{t^2}{e^t+e^{-t}+2},$
$$\begin{aligned}
I_{22}
=&2\int^{\log n}_0\int^{\log n}_0f(u-v)dvdu+2\int^{\log n}_0\int^{\log n}_0h(u-v)dvdu \\
=&4\int_0^{\log n}(f(t)+h(t))(\log n-t)dt\\
=&4\log n\int_0^{\log n}(f(t)+h(t))dt-4\int_0^{\log n}(f(t)+h(t))tdt.
\end{aligned}$$
It is easy to verify that
the following integrals converge
$$A:=\int_0^{+\infty}(f(t)+h(t))dt<\infty,\quad B:=\int_0^{+\infty}(f(t)+h(t))tdt<\infty.$$
To be more precise, we claim $A=\frac{\pi^2}{4}.$
Indeed, let $w=e^{-t}$, we have
$$\begin{aligned}A=&\int_0^{+\infty}(f(t)+h(t))dt\\
=&\int_0^1\left(\frac{2}{(1-w^2)^2}-\frac{1}{1-w^2}\right)(\log w)^2dw\\
=&\sum_{j=0}^{\infty}\left[(2j+1)\int_0^1w^{2j}(\log w)^2dw\right].
\end{aligned}$$
Here, let $\mu=-\log w$, $s=(2j+1)\mu$, then
$$\int_0^1w^{2j}(\log w)^2dw=\int_0^{+\infty}e^{-(2j+1)\mu} \mu^2d\mu=\frac{1}{(2j+1)^3}\int_0^{+\infty}e^{-s}s^2ds=\frac{\Gamma(3)}{(2j+1)^3}.$$
Hence
$$A=2\sum_{j=0}^{\infty}\frac{1}{(2j+1)^2}=\frac{\pi^2}{4}.$$
Thus
$$\begin{aligned}I_{22}
=&\pi^2\log n -4\log n\int_{\log n}^{+\infty}(f(t)+h(t))dt-4B+4\int_{\log n}^{+\infty}(f(t)+h(t))tdt.
\end{aligned}$$
Now, we estimate the remaining items.
Since $f(t)+h(t)\leq Ct^2 e^{-t}$ for large $t$, we have
$$\log n\int_{\log n}^{+\infty}(f(t)+h(t))dt\leq C\log n \int_{\log n}^{+\infty}t^2 e^{-t}dt=\log n \cdot O\left(\frac{(\log n)^2}{n} \right)=o(1),$$
and
$$\int_{\log n}^{+\infty}(f(t)+h(t))tdt\leq C\log n \int_{\log n}^{+\infty}t^3 e^{-t}dt= O\left(\frac{(\log n)^3}{n}\right)=o(1),$$
as $n\to\infty$.
Therefore, for large $n$
\begin{equation}\label{e0701-e1}
I_{22}\leq \pi^2\log n.
\end{equation}

(iv) Now, we estimate
$$\begin{aligned}
I_{23}=I_{32}=&2\int_{\frac{1}{n}}^1\int_{-\frac{1}{n}}^{\frac{1}{n}} \frac{\left(\log  (nx)\right)^2}{(x-y)^2} d ydx\\
=&\frac{4}{n}\int_{\frac{1}{n}}^1\frac{(\log(nx))^2}{x^2-\frac{1}{n^2}}dx\\
=&4\int_1^n\frac{(\log t)^2}{t^2-1}dt\\
=&4\int_1^2\frac{(\log t)^2}{t^2-1}dt+4\int_2^n\frac{(\log t)^2}{t^2-1}dt\\
\leq&4C+\frac{16}{3}\int_2^n\frac{(\log t)^2}{t^2}dt\\
=&4C+\frac{16}{3}\left(-\frac{(\log n)^2}{n}-\frac{2\log n}{n}-\frac{2}{n}+\frac{(\log 2)^2}{2}+\log (2e)\right)\\
\leq &4C,\quad\text{for}~n~\text{sufficiently~large},
\end{aligned}$$
where,
$$C=\int_1^2\frac{(\log t)^2}{t^2-1}dt=\int_0^1\frac{(\log (1+s))^2}{s(s+2)}ds<\infty$$
since $$\lim\limits_{s\to0}\frac{(\log (1+s))^2 }{s(s+2)}\cdot s^{1/2}=0.$$

In summary, from (i) to (iv), it can be concluded that
$$
\begin{aligned}
\relax[w_n]^2_{H^{\frac{1}{2}}(\mathbb{R})}=
&\int_{{\mathbb{R}}} \int_{\mathbb{R}} \frac{(w_n(x)-w_n(y))^2}{(x-y)^2} d y d x\\
\leq  &\left[C+\pi^2\log n+4(\log n)^2\log\left(1+\frac{2}{n-1}\right)\right] A_n^2\\
=&\left[C+\pi^2\log n+4(\log n)^2\log\left(1+\frac{2}{n-1}\right)\right]\frac{1}{\pi}\frac{1}{\log n}\left(1-\frac{1}{\log n}\right)\\
\to&\pi,\quad\text{as}~n\to\infty.
\end{aligned}
$$
Then by Proposition \ref{A.1}
$$\left|(-\Delta)^\frac{1}{4} w_n\right|^2_{L^2(I)}\leq
\left|(-\Delta)^\frac{1}{4} w_n\right|^2_{L^2(\mathbb{R})}=\frac{1}{2}C_{1,\frac{1}{2}}[w_n]^2_{H^{\frac{1}{2}}(\mathbb{R})}\leq1$$ for $n$ sufficiently large and thus $w_n \in E$. Here $$ C_{1,\frac{1}{2}}= \frac{2 \cdot \frac{1}{2} \cdot \Gamma\left(\frac{1}{2} + \frac{1}{2}\right)}{\pi^{\frac{1}{2}} \cdot \Gamma(1-\frac{1}{2})} = \frac{ \Gamma\left(1\right)}{\pi^{\frac{1}{2}} \cdot \Gamma(\frac{1}{2})}= \frac{1}{\pi} $$ is the normalization constant of the $1/2$-Laplacian.

Now, we prove $\Phi(w_n)\to\infty$. Let $v_n:=$ $G\left(w_n\right) \in L^{\infty}\left(I\right)$ and therefore
$$
\Phi_{ \pm}\left(w_n\right)<\infty \quad \text { for } n \in \mathbb{N}.
$$
By (\ref{eq:3.4}) we have, for $n$  sufficiently large,
$$
v_n \geq c_1^{-1}\left(\frac{\log n}{\pi}-\frac{1}{\pi}\right)^{-\frac{\gamma}{2}} e^{\pi\left(\frac{\log n}{ \pi}-\frac{1}{\pi}\right)} \geq c_2(\log n)^{-\frac{\gamma}{2}} n\quad \text { on } I_{\frac{1}{n}}(0)
$$
with a constant $c_2>0$. We derive that
$$
\begin{aligned}
&\int_{\mathbb{R}}\int_{\mathbb{R}}\log\frac{1}{|x-y|}~v_n(x)v_n(y)dxdy\\ \geq &2 \int_0^{\frac{1}{n}} v_n(r) \log \frac{1}{r} \int_0^r v_n(\rho) d \rho d r \\
\geq &2c_2^2(\log n)^{-\gamma} n^2 \int_0^{\frac{1}{n}} r \log \frac{1}{r} d r\\
 =&-2\left.c_2^2(\log n)^{-\gamma} n^2\left(\frac{r^2}{2} \log r-\frac{r^2}{4}\right)\right|_0 ^{\frac{1}{n}} \\
 \geq &\frac{1}{4} c_2^2(\log n)^{1-\gamma}
\end{aligned}
$$
so that $$\Phi(w_n)=\int_{\mathbb{R}}\int_{\mathbb{R}}\log\frac{1}{|x-y|}~v_n(x)v_n(y)dxdy \rightarrow \infty$$ as $n \rightarrow \infty$ since $\gamma<1$.
\end{proof}

Next we prove the existence of a maximizer $u \in E$ for $m(G)$ under the assumptions of Proposition \ref{pro:3.2}. We first prove the following convergence result.

\begin{proposition}\label{pro:3.4}
    Suppose that $G$ satisfies $(A)$, $G$ has at most $\gamma$-critical growth for some $\gamma >1$, and  $$G(s)-G(0) \leq c \, \left(e^{\pi s^{2}}-1\right)$$ for $s \in \mathbb{R}$ with  some $c>0$ and $\gamma>1$. Let $\{u_n\}$ be a sequence in $E$ with $u_n \rightharpoonup 0$ in $H_0^\frac{1}{2}\left(I\right)$. Then

\begin{equation}\label{eq:3.5}
    \Phi\left(u_n\right) \rightarrow \Phi(0) \quad \text { as } n \rightarrow \infty.
\end{equation}

\end{proposition}

\begin{proof} Without loss of generality, we assume $I=(-1, 1)$.
     Since $u_n \rightharpoonup 0$ in $H_0^\frac{1}{2}\left(I\right)$ and $H_0^\frac{1}{2}\left(I\right)$ is compactly embedded into $L^p\left(I\right)$ for $2<p<\infty$, we have
\begin{equation}\label{eq:3.6}
    u_n \rightarrow 0 \quad \text { in } L^p\left(I\right) \text { for } 2<p<\infty.
\end{equation}
And since $u_n$ are radial and non-increasing in the radial coordinate,
\begin{equation}\label{eq:3.7}
   \begin{aligned}
    u_n \rightarrow 0 \quad~\text{uniformly~in}~[\delta, 1] ~\text{for~ every} ~\delta\in(0,1).
   \end{aligned}
\end{equation}
Denote $\tilde{G}=G-G(0)$, then $\tilde{G}$  satisfies $(A)$, $ \tilde{G}(0)=0$  and
\begin{equation}\label{eq:3.8}
    \tilde{G}(s)\leq c_1\frac{\pi s^2}{(1+|s|)^\gamma},~~ \tilde{G}(s)\leq  c_2\left(e^{\pi s^{2}}-1\right) \leq  c_2e^{\pi s^2}, \quad \text { for }~ s \in \mathbb{R}.
\end{equation}
Let
$$
v_n:=1_{I} \tilde{G}\left(u_n\right) \quad \text {for}~ n \in \mathbb{N},
$$
we then get
\begin{align*}
     &\Phi\left(u_n\right)\\
 =&\int_{\mathbb{R}}\int_{\mathbb{R}}\log\frac{1}{|x-y|}~v_n(x)v_n(y)dxdy+2\int_{\mathbb{R}}\int_{\mathbb{R}}\log\frac{1}{|x-y|}~1_{I} G(0)v_n(y)dxdy\\
&+\int_{\mathbb{R}}\int_{\mathbb{R}}\log\frac{1}{|x-y|}~1_{I} G(0)1_{I}G(0)dxdy\\
=&\int_{\mathbb{R}}\int_{\mathbb{R}}\log\frac{1}{|x-y|}~v_n(x)v_n(y)dxdy+2\int_{\mathbb{R}}\int_{\mathbb{R}}\log\frac{1}{|x-y|}~1_{I}G(0)v_n(y)dxdy+\Phi(0).
\end{align*}
To verify (\ref{eq:3.5}), it is suffice to prove
\begin{equation}\label{eq:3.11}
I_n:=\int_{\mathbb{R}}\int_{\mathbb{R}}\log\frac{1}{|x-y|}~v_n(x)v_n(y)dxdy \rightarrow 0,
\end{equation} and \begin{equation}\label{0704-eq:3.11}
II_n:=\int_{\mathbb{R}}\int_{\mathbb{R}}\log\frac{1}{|x-y|}~1_{I}G(0)v_n(y)dxdy \rightarrow 0,\quad \text { as } n \rightarrow \infty.
\end{equation}

(i) To see (\ref{eq:3.11}), we note that for every $\delta \in(0,1)$ we have
$$\begin{aligned}
I_n
=&2\int_0^1 v_n(r) \log \frac{1}{r} \int_0^r v_n(\rho) d \rho d r\\
=&2\int_\delta^1 v_n(r) \log \frac{1}{r} \int_0^r v_n(\rho) d \rho d r+2\int_0^\delta v_n(r) \log \frac{1}{r} \int_0^r v_n(\rho) d \rho d r\\
:=&2(I_n^{(1)}+I_n^{(2)}).\end{aligned}
$$
By (\ref{eq:1.2}), (\ref{eq:3.8}) and (\ref{eq:3.7})
\begin{equation}\label{eq:3.12}
\begin{aligned}
I_n^{(1)} &  \leq c_1 \int_\delta^1 v_n(r) \log \frac{1}{r} \int_0^1  e^{\pi u_n^2(\rho)} d \rho d r \\
& \leq c_1c(I)\int_\delta^1 v_n(r) \log \frac{1}{r} d r \rightarrow 0 \quad \text { as } n \rightarrow \infty.
\end{aligned}
\end{equation}
To estimate $I_n^{(2)}$
we fix $\varepsilon \in\left(0, \frac{1}{2\pi}\right)$ and define, for any $n \in \mathbb{N}$,
$$
\begin{aligned}
A_n^{+}:=&\left\{r \in(0,1]: u_n(r) \geq \sqrt{\varepsilon(-\log r)}\right\}, \\ A_n^{-}:=&\left\{r \in(0,1]: u_n(r)<\sqrt{\varepsilon(-\log r)}\right\}.
\end{aligned}
$$
By (\ref{eq:3.8}), we have
\begin{equation}\label{eq:3.14}
    v_n(r) \leq c_1 \frac{e^{\pi u_n^2(r)}}{(1+\sqrt{\varepsilon(-\log r)})^\gamma} \quad \text { for } r \in A_n^{+},
\end{equation}
and
\begin{equation}\label{eq:3.13}
    v_n(r) \leq c_1 e^{\pi u_n^2(r)} \leq \frac{c_1}{r^{\pi \varepsilon}} \quad \text { for } r \in A_n^{-}.
\end{equation}
We now write
$$
I_n^{(2)}=\int_{A_n^{-} \cap(0, \delta)} v_n(r) \log \frac{1}{r} \int_0^r v_n(\rho) d \rho d r+\int_{A_n^{+} \cap(0, \delta)} v_n(r) \log \frac{1}{r} \int_0^r v_n(\rho) d \rho d r,
$$
where, by (\ref{eq:1.2}) and \eqref{eq:3.13},
\begin{equation}\label{eq:3.15}
    \begin{aligned}
& \int_{A_n^{-} \cap(0, \delta)} v_n(r) \log \frac{1}{r} \int_0^r v_n(\rho) d \rho d r \\\leq &c_1^2 \int_0^\delta r^{-\pi \varepsilon} \log \frac{1}{r} \int_0^1 e^{\pi u_n^2(\rho)} d \rho d r \\
 \leq &c_1^2c(I) \int_0^\delta r^{-\pi \varepsilon} \log \frac{1}{r} d r\\
=&c_1^2c(I)\left(\frac{\delta^{1-\pi \varepsilon}}{ (1-\pi \varepsilon)^2}-\frac{\delta^{1-\pi \varepsilon} \log \delta}{1-\pi \varepsilon}\right)
\end{aligned}
\end{equation}
for all $n \in \mathbb{N}$, where the constant $c(I)$ is the supremum in the fractional Trudinger-Moser inequality (\ref{eq:1.2}). And
$$
\begin{aligned}
& \int_{A_n^{+} \cap(0, \delta)} v_n(r) \log \frac{1}{r} \int_0^r v_n(\rho) d \rho d r \\
=&\int_{A_n^{+} \cap(0, \delta)} v_n(r) \log \frac{1}{r} \int_{A_n^{+} \cap(0, r)} v_n(\rho) d \rho d r+\int_{A_n^{+} \cap(0, \delta)} v_n(r) \log \frac{1}{r} \int_{A_n^{-} \cap(0, r)} v_n(\rho) d \rho d r,
\end{aligned}
$$
where
\begin{equation}\label{eq:3.16}
    \begin{aligned}
& \int_{A_n^{+} \cap(0, \delta]} v_n(r) \log \frac{1}{r} \int_{A_n^{+} \cap[0, r]} v_n(\rho) d \rho d r \\
 \leq &c_1^2 \int_{A_n^{+} \cap(0, \delta]}  \frac{e^{\pi u_n^2(r)}}{(1+\sqrt{\varepsilon(-\log r)})^\gamma} \log \frac{1}{r} \int_{A_n^{+} \cap[0, r]} \frac{ e^{\pi u_n^2(\rho)}}{(1+\sqrt{\varepsilon(-\log \rho)})^\gamma} d \rho d r \\
 \leq &c_1^2 \int_{A_n^{+} \cap(0, \delta]}(-\log r)(1+\sqrt{\varepsilon(-\log r)})^{ -2\gamma} e^{\pi u_n^2(r)} \int_0^r e^{\pi u_n^2(\rho)} d \rho d r \\
 \leq &\frac{c_1^2}{\varepsilon}(1+\sqrt{\varepsilon(-\log \delta)})^{2-2\gamma}\left(\int_0^1 e^{\pi u_n^2(r)} d r\right)^2 \\
 \leq &\frac{\left(c_1 c\left(I\right)\right)^2}{\varepsilon}\left( \sqrt{\varepsilon(-\log \delta)}\right)^{2-2\gamma}
\end{aligned}
\end{equation}
again by (\ref{eq:1.2}). Here we used the assumption $\gamma>1$. Finally, we have
\begin{equation}\label{eq:3.17}
    \begin{aligned}
& \int_{A_n^{+} \cap[0, \delta]} v_n(r) \log \frac{1}{r} \int_{A_n^{-} \cap[0, r]} v_n(\rho) d \rho d r \\\leq& c_1^2 \int_{A_n^{+} \cap[0, \delta]} e^{\pi u_n^2(r)} \log \frac{1}{r} \int_0^r \rho^{- \pi \varepsilon} d \rho d r \\
 \leq &\frac{c_1^2}{1-\pi \varepsilon} \int_0^\delta r^{1-\pi \varepsilon} e^{\pi u_n^2(r)} \log \frac{1}{r} d r \\
 \leq& c_1^2 C_\delta \int_0^1 e^{\pi u_n^2(r)} d r \leq c_1^2 C_\delta c\left(I\right)
\end{aligned}
\end{equation}
by (\ref{eq:1.2}) with $$C_\delta=\sup _{r \in[0, \delta]} 2r^{1-\pi \varepsilon} \log \frac{1}{r}.$$ Since, as $\varepsilon \in\left(0, \frac{1}{2\pi}\right)$ and $\gamma>1$, the RHS of (\ref{eq:3.15}), (\ref{eq:3.16}), and (\ref{eq:3.17}) tends to zero as $\delta \rightarrow 0^{+}$, we infer that
$$
\lim _{\delta \rightarrow 0} \sup _{n \in \mathbb{N}} I_n^{(2)}=0.
$$
Combining this with (\ref{eq:3.12}), we infer (\ref{eq:3.11}), as claimed.

(ii) We prove \eqref{0704-eq:3.11}, i.e., $II_n\to0.$
By Corollary \ref{cor:2.8}-(i), we have
$$
II_n=\int_{\mathbb{R}}\int_{\mathbb{R}}\log\frac{1}{|x-y|}~1_{I} G(0)v_n(y)dxdy=G(0) \int_0^1 v_n(r) f(r) d r,
$$ where $$f(r)=\log \frac{1}{r} \int_0^r d \rho+\int_r^1 \left(\log \frac{1}{\rho}\right) d \rho=1-r.$$
Then
$$0\leq II_n\leq G(0)\int_0^1 v_n f(r)dr=G(0)\int_0^1\left[G(u_n(r))-G(0)\right](1-r)dr.$$
To show \eqref{0704-eq:3.11}, by the assumption \eqref{20250704-e1} about $G$ it is enough to show
$$D_n:=\int_0^1 \left(e^{\pi u_n^2(r)}-1\right)(1-r)dr\to0.$$
Now, we use the $L^p$
 convergence of $u_n$ to select truncation parameters.
By $|u_n|_p\to0$, we choose $0<\alpha<1$ and define
$$\delta_n:=|u_n|_p^\alpha,\quad S_n:=\{r\in(0,1): |u_n(r)|>\delta_n\}.$$
Then, $D_n$ can be divided into two parts:
$$D_n=\int_{S_n^c}\left(e^{\pi u_n^2(r)}-1\right)(1-r)dr+\int_{S_n}\left(e^{\pi u_n^2(r)}-1\right)(1-r)dr:=D_n^{(1)}+D_n^{(2)}.$$
By Chebyshev inequality,
\begin{equation}\label{Chebyshev}
|S_n|\leq \delta_n^{-p}|u_n|_p^p=|u_n|_p^{p(1-\alpha)}\to0.
\end{equation}

First, we estimate $D_n^{(1)}$. Since $\delta_n=|u_n|_p^\alpha\to0$, we have for $x\in S_n^c$, $|u_n(x)|\leq \delta_n<1$ for large $n$, and $\pi u_n^2\leq\pi \delta_n^2\leq1.$ Hence, by the elementary inequality $e^{ x}-1\leq ex$ for any $x\leq1$,
$$e^{\pi u_n^2(r)}-1\leq  e\pi u_n^2,\quad \text{for}~\pi u_n^2\leq1.$$
 Therefore,
\begin{equation}\label{929}
    D_n^{(1)}\leq c\int_{S_n^c}u_n^2(1-r)dr\leq c'|u_n|_p^2\to0.
\end{equation}

Now, we estimate $D_n^{(2)}$.
For any $M>0$ we define $$T_n(M):=\left\{r: e^{\pi u_n^2(r)}-1>M\right\}.$$
By Chebyshev inequality and fractional Trudinger-Moser inequality \eqref{eq:1.1},
$$|T_n(M)|\leq \frac{1}{M}\int_0^1\left(e^{\pi u_n^2}-1\right)dr\leq\frac{c(I)}{M}.$$
We divide $D_n^{(2)}$ into two parts:
$$D_n^{(2)}=\int_{S_n\cap T_n(M)}\left(e^{\pi u_n^2}-1\right)(1-r)dr+\int_{S_n\setminus T_n(M)}\left(e^{\pi u_n^2}-1\right)(1-r)dr.$$
For $r\in S_n\setminus T_n(M)$, we have
$e^{\pi u_n^2}-1\leq M$, then by \eqref{Chebyshev}
$$\int_{S_n\setminus T_n(M)}\left(e^{\pi u_n^2}-1\right)(1-r)dr\leq M|S_n|\to0.$$
For $r\in S_n\cap T_n(M)$, by fractional Trudinger-Moser inequality \eqref{eq:1.1}
$$\int_{S_n\cap T_n(M)}\left(e^{\pi u_n^2}-1\right)(1-r)dr\leq C|T_n(M)|\leq C\frac{c(I)}{M}.$$
Therefore, by the arbitrariness of $M$, we have $D_n^{(2)}\to0$.
Combing this with \eqref{929}, we get $D_n\to0$ and hence \eqref{0704-eq:3.11}.

Combining \eqref{eq:3.11} in (i) and \eqref{0704-eq:3.11} in (ii), \eqref{eq:3.5} holds.
\end{proof}

\begin{proposition}\label{pro:3.5}
    Suppose that $G$ satisfies $(A)$ and has at most $\gamma$-critical growth with $\gamma>1$. Then the value $m(G)<\infty$ is achieved by a function $u \in E$.
\end{proposition}
\begin{proof}
    Let $\{u_n\}$ be a maximizing sequence in $E$  for $m(G)$.  Since $E$  is bounded in $H_0^\frac{1}{2}\left(I\right)$, we may assume that
$$
u_n \rightharpoonup u \in H_0^\frac{1}{2}\left(I\right) \quad \text { with } u \in E.
$$
By the concentration-compactness principle \cite[Theorem 1.6]{19},  two scenarios may occur:

\textbf{Case 1: $u = 0$.} \\
Proposition \ref{pro:3.4} implies that
$$
m(G)=\lim _{n \rightarrow \infty} \Phi\left(u_n\right)=\Phi(0).
$$
However, Remark \ref{re:2.9} states $\Phi(0) = 0$, contradicting $m(G) > 0$.
Thus this case is impossible.

\textbf{Case 2: $u \neq 0$}, and there exists $t > 0$ such that
$e^{(\pi + t) u_n^2}$ is bounded in $L^1(I)$. \\
Consequently, by the fractional Trudinger-Moser inequality,
\begin{equation}\label{eq:exp_conv}
e^{\pi u_n^2} \to e^{\pi u^2} \quad \text{in } L^1(I).
\end{equation}
Define $v_n:=1_{I} G\left(u_n\right)$ for $n \in \mathbb{N}$ and $v:=$ $1_{I} G(u)$. By the hypothesis of  Case 2, $ \int_{I} e^{(\pi+t) u_n^2} d x$ is bounded for some $t>0$, thus H{\"o}lder's inequality shows $v_n$ is bounded in $L^{s_0}\left(\mathbb{R}\right)$ with $s_0=1+\frac{t}{\pi}>1$. Moreover, interpolation with
$$
v_n \rightarrow v \quad \text { in } L^1\left(\mathbb{R}\right),
$$
yields that
$$
v_n \rightarrow v \quad \text { in } L^s\left(\mathbb{R}\right) \quad \text { for } 1 \leq s<s_0.
$$
Consider the difference:
$$
\begin{aligned}
& \frac{\Phi\left(u_n\right)-\Phi(u)}{2}\\
=&\int_0^1 v_n(r) \log \frac{1}{r} \int_0^rv_n(r) d \rho d r-\int_0^1 v(r) \log \frac{1}{r} \int_0^r v(\rho) d \rho d r \\
=&\int_0^1 v_n(r) \log \frac{1}{r} \int_0^r \left[v_n(\rho)-v(\rho)\right] d \rho d r+\int_0^1 \left[v_n(r)-v(r)\right] \log \frac{1}{r} \int_0^r v(\rho) d \rho d r,
\end{aligned}
$$
where, for fixed $s \in\left(1, s_0\right)$, by H{\"o}lder's inequality and the boundedness of $\{v_n\}$ in $L^1(\mathbb{R})$,
$$
\begin{aligned}
& \left|\int_0^1 v_n(r) \log \frac{1}{r} \int_0^r \left[v_n(\rho)-v(\rho)\right] d \rho d r\right| \\
\leq& \left|v_n-v\right|_s \int_0^1 \left|I_r\right|^{\frac{1}{s^{\prime}}} v_n(r) \log \frac{1}{r} d r \\
 \leq &\left|v_n-v\right|_s\int_0^1 (2r)^{\frac{1}{s^{\prime}}} v_n(r) \log \frac{1}{r} d r \\\leq& C\left|v_n-v\right|_s\left|v_n\right|_1 \rightarrow 0 \quad \text { as } n \rightarrow \infty,
\end{aligned}
$$
 and
$$
\begin{aligned}
& \left|\int_0^1 \left[v_n(r)-v(r)\right] \log \frac{1}{r} \int_0^r v(\rho) d \rho d r\right| \\
\leq&|v|_s \int_0^1 \left|I_r\right|^{\frac{1}{s^{\prime}}}\left[v_n(r)-v(r)\right] \log \frac{1}{r} d r \\
 \leq& C|v|_s\int_0^1 (2r)^{\frac{1}{s^{\prime}}}\left[v_n(r)-v(r)\right] \log \frac{1}{r} d r \\
 \leq& C|v|_s\left|v_n-v\right|_1 \rightarrow 0 \quad \text { as } n \rightarrow \infty.
\end{aligned}
$$
Combining above estimates, we conclude
$$
m(G)=\lim_{n \rightarrow \infty} \Phi\left(u_n\right)=\Phi(u),
$$
so $m(G)$ is attained at $u \in E$.
\end{proof}

The proof of Theorem \ref{th:1.2} is now completed by the following lemma.
\begin{lemma}\label{le:3.6}
     Let $G$ satisfy $(A)$, and let $u \in E$ be a maximizer for $\left.\Phi\right|_{E}$. Then, up to a change of sign, we have $u \in E$, and $u$ is strictly positive in $I$.
\end{lemma}

\begin{proof}
    We already know that $u \not \equiv 0$. We first assume that $u \in E$, and we suppose by contradiction that there exists $\tau \in(0,1)$ with $u(r)>0$ for $r \in(0, \tau)$ and $u(r)=0$ for $r \in[\tau, 1)$. Define $\tilde{u}(r)=u(\tau r)$, then $\tilde{u} \in H_0^\frac{1}{2}\left(I\right)$ and $|(-\Delta)^{\frac{1}{4}}\tilde{u}|_{2}=|(-\Delta)^{\frac{1}{4}}{u}|_{2}$,  therefore $\tilde{u} \in E$. Moreover, with $v:=G(u)$ we have
$$
\Phi(\tilde{u})=2\int_0^1 v(\tau r) \log \frac{1}{r} \int_0^r v(\tau \rho) d \rho d r,
$$
where
$$
v(\tau r) \geq v(r) ~ \text { for } r \in(0, \tau), \quad \text { and } ~ v(\tau r)=G(u(\tau r))>G(0)=v(r) ~\text { for } r \in[\tau, 1),
$$
since $G$ is strictly increasing on $[0, \infty)$ by assumption $(A)$. It thus follows that
$$
\Phi(\tilde{u})>2 \int_0^1 v(r) \log \frac{1}{r} \int_0^r v(\rho) d \rho d r=\Phi(u)
$$
contrary to the maximizing property of $u$. We thus conclude that $u$ is strictly positive in $I$.
 The proof is thus finished.
\end{proof}

\section{Maximization problem in the entire space  }

In this section we complete the proof of Theorem \ref{th:1.4}. We recall that $E_{\infty}:=\left\{u \in H^\frac{1}{2}\left(\mathbb{R}\right):\|u\| \leq1\right\}$.
Throughout this section, we assume that $G: \mathbb{R} \rightarrow \mathbb{R}$ satisfies assumption $\left(A_1\right)$.

 \begin{lemma}\label{le:4.1}
     Suppose that $G$ has at most $\gamma$-critical growth for some $\gamma \geq1$. Then
$$
m_{\infty}^{+}(G)=\sup _{E_{\infty}} \Psi_{+}<\infty.
$$
\end{lemma}
\begin{proof}
    By assumption, there exists a constant $c_1>0$ with
\begin{equation}\label{eq:4.1}
    G(s) \leq \widetilde{G}(s):=c_1 \frac{e^{\pi s^2}}{\sqrt{1+s^2}} \quad \text { for~every } s \in \mathbb{R}.
\end{equation}
We note that $\widetilde{G}$ satisfies assumption $(A)$. Moreover, we note the property that
\begin{equation}\label{eq:4.2}
    \widetilde{G}(\sqrt{t+s})=c_1 \frac{e^{\pi(t+s)}}{\sqrt{1+(s+t)}} \leq c_1 e^{\pi s} \frac{e^{\pi t}}{\sqrt{1+t}}=e^{\pi s} \widetilde{G}(\sqrt{t}) \quad \text { for } t, s \geq 0.
\end{equation}
As we claimed in Lemma \ref{le:2.3} (i), we find that $\Psi_{+}(u^*)\geq\Psi_{+}(u^*)$, therefore for the sake of argument, let's assume $u\in E_{\infty}$. Since
 $$
 \begin{aligned}
     |u|^2_{2}\geq\int_{0}^{d}(u(s))^2\,ds\geq (u(d))^2d,
 \end{aligned}
 $$
we have the estimate
\begin{equation}\label{eq:4.3}
   u(d) \leq |u|_{2} d^{-\frac{1}{2}} \leq\|u\| d^{-\frac{1}{2}} \leq d^{-\frac{1}{2}}.
\end{equation}
Let $v=G(u)$. We write
$$
v=G(u)=G\left(u 1_{\mathbb{R} \backslash I}\right)+G\left(u 1_{I}\right)=: v_1+v_2.
$$
Since $G$ is nonnegative and $G(t)=O(|t|)$ as $t \rightarrow 0$, it follows that
$$
0 \leq v(d) \leq c_2|u(d)| \quad \text { for } d \geq 1
$$
with a constant $c_2>0$ and therefore,
\begin{equation}\label{eq:4.4}
    \left|v_1\right|_2 \leq c_2|u|_2 \leq c_2\|u\| \leq c_2.
\end{equation}
To estimate $v_2$, we now consider the radial function
$$
w:=(1+|u|_2^2)^{\frac{1}{2}}[u-u(1)]^{+} .
$$
As we proved in Lemma \ref{le:2.5.1}, we have $w\in H_0^\frac{1}{2}\left(I\right)$, meanwhile, since $u\in E_{\infty}$,
$$
\begin{aligned}
\left\|(-\Delta)^{\frac{1}{4}} w\right\|_{L^2(I)}^2
& \leq(1+|u|_2^2)\left\|(-\Delta)^{\frac{1}{4}} u\right\|_{L^2(\mathbb{R})}^2\leq\|u\|\leq1,
\end{aligned}
$$
which implies $w\in E_\infty$. Using \eqref{eq:2.5.1} and the fractional Trudinger-Moser inequality (\ref{eq:1.2}), we infer
\begin{equation}\label{eq:4.6}
    \left|v_2\right|_1 \leq c_1 \int_{I} e^{\pi w^2} d x \leq c_1 c\left(I\right)=: c_3.
\end{equation} \\
By (\ref{eq:2.5.1}) and (\ref{eq:4.2}), we also have
$$
v_2 \leq \widetilde{G}\left(\sqrt{u^2}\right) \leq \widetilde{G}\left(\sqrt{w^2+2}\right) \leq  e^{2\pi} \widetilde{G}(w) \quad \text { in } I .
$$
We now have
\begin{equation}\label{eq:4.8}
\begin{aligned}
    \Psi_{+}(u)=&\int_{\mathbb{R}}\int_{\mathbb{R}}\log^+\frac{1}{|x-y|}~v(x)v(y)dxdy
\\=&\int_{\mathbb{R}}\int_{\mathbb{R}}\log^+\frac{1}{|x-y|}~v_1(x)v_1(y)dxdy+\int_{\mathbb{R}}\int_{\mathbb{R}}\log^+\frac{1}{|x-y|}~v_2(x)v_2(y)dxdy
\\&+2 \int_{\mathbb{R}}\int_{\mathbb{R}}\log^+\frac{1}{|x-y|}~v_1(x)v_2(y)dxdy,
\end{aligned}
\end{equation}
where
$$
\begin{aligned}
&\int_{\mathbb{R}}\int_{\mathbb{R}}\log^+\frac{1}{|x-y|}~v_2(x)v_2(y)dxdy\\ \leq& e^{4\pi} \int_{\mathbb{R}}\int_{\mathbb{R}}\log^+\frac{1}{|x-y|}~\widetilde{G}(w(x))\widetilde{G}(w(y))dxdy\\
\leq &e^{4 \pi} m_1^{+}(\widetilde{G})=:c_4<\infty
\end{aligned}
$$
by Proposition \ref{pro:3.2}. Moreover, by Young's inequality, (\ref{eq:4.4}) and (\ref{eq:4.6}) we have
$$
\int_{\mathbb{R}}\int_{\mathbb{R}}\log^+\frac{1}{|x-y|}~v_1(x)v_2(y)dxdy \leq\left|v_1\right|_2\left|\log ^{+} \frac{1}{|\cdot|}\right|_2\left|v_2\right|_1 \leq c_2\left|\log ^{+} \frac{1}{|\cdot|}\right|_2 c_3=: c_5 .
$$
Moreover, again by Young's inequality,
$$
\int_{\mathbb{R}}\int_{\mathbb{R}}\log^+\frac{1}{|x-y|}~v_1(x)v_1(y)dxdy \leq\left|v_1\right|_2^2\left|\log ^{+} \frac{1}{|\cdot|}\right|_1 \leq c_2^2\left|\log ^{+} \frac{1}{|\cdot|}\right|_1=: c_6 .
$$
Inserting these estimates in (\ref{eq:4.8}), we deduce that $\Psi_{+}(u) \leq c_4+c_6+2 c_5<\infty$. Thus the claim is proved.
\end{proof}

    From Lemma \ref{le:4.1} it follows that
$$
m_{\infty}(G)=\sup _{E_{\infty}} \Psi \leq m_{\infty}^{+}(G)<\infty.
$$

Our next aim is to prove the existence of maximizers for $m_{\infty}(G)$.
\begin{proposition}\label{pro:4.3}
    Suppose that $G$ satisfies $(A)$, $G$ has at most $\gamma$-critical growth for some $\gamma >1$, and \begin{equation*}\label{74}
        G(s) \leq c \, \left(e^{\pi s^{2}}-1\right) \quad\text{for}~ s \in \mathbb{R}
    \end{equation*}  with  some $c>0$. Then the value $m_{\infty}(G)<\infty$ is attained by a function $u \in E_{\infty}$.
\end{proposition}

\begin{proof}
     Let $\{u_n\}$ be a maximizing sequence in $E_{\infty}$ for $m_{\infty}(G)$. Moreover, we have $b_{-}\left(G\left(u_n\right)\right)<\infty$ for every $n \in \mathbb{N}$. Hence, by Corollary \ref{cor:2.8}, we have
$$
\Psi\left(u_n\right)=\int_{\mathbb{R}}\int_{\mathbb{R}}\log\frac{1}{|x-y|}~G(u(x))G(u(y))dxdy=2\left(\Psi_2\left(u_n\right)-\Psi_1\left(u_n\right)\right) ~ \text { for } n \in \mathbb{N}
$$
with the nonnegative functionals $\Psi_1, \Psi_2: E_{\infty} \rightarrow[0, \infty)$ given by
$$
\Psi_1(u)=\int_1^{\infty} G(u(r)) \log r \int_0^r G(u(\rho)) d \rho d r,$$ $$ \Psi_2(u)=\int_0^1 G(u(r)) \log \frac{1}{r} \int_0^r G(u(\rho)) d \rho d r .
$$
Since $\{u_n\}$ is bounded in $E_\infty\subset H^\frac{1}{2}\left(\mathbb{R}\right)$, we may assume that
\begin{equation}\label{eq:4.10}
    u_n \rightharpoonup u\in E_{\infty} \subset H^\frac{1}{2}\left(\mathbb{R}\right).
\end{equation}

We first claim that
\begin{equation}\label{eq:4.11}
    \Psi_2(u) \geq \limsup _{n \rightarrow \infty} \Psi_2\left(u_n\right).
\end{equation}
By the concentration-compactness principle \cite[Theorem 1.6]{19},
there are two possibilities. Either
(i) $u=0$, or
(ii) $u \neq 0$, and $\int_{I} e^{(\pi+t) u_n^2} d x$ is bounded for some $t>0$ and thus
\begin{equation}\label{eq:4.12}
    e^{\pi u_n^2} \rightarrow e^{\pi u^2} \quad \text { in } L^1\left(I\right) .
\end{equation}

Case (i). As in the proof of Proposition \ref{pro:3.4}, we then deduce that
\begin{equation}\label{eq:4.13}
    u_n(r) \rightarrow 0 \quad \text { as } n \rightarrow \infty \text { for all } r>0 .
\end{equation}
Moreover, there exists a constant $c>0$ with
\begin{equation}\label{eq:4.14}
    G(s) \leq \widetilde{G}(s):=c \frac{e^{\pi s^2}}{\left(1+s^2\right)^{\frac{\gamma}{2}}} \quad \text { for } s \in \mathbb{R}.
\end{equation}
We may assume that $2 \pi>\gamma>1$ from now on. Then
$$
\widetilde{G}^{\prime}(s)=c \frac{e^{\pi s^2}}{\left(1+s^2\right)^{\frac{\gamma}{2}+1}}\left(2 \pi s\left(1+s^2\right)-\gamma s\right)>0 \quad \text { for } s>0,
$$
and therefore $\widetilde{G}$ satisfies assumption $(A)$. Let
$$
w_n:=(1+|u|_2^2)^{\frac{1}{2}}\left[u_n-u_n(1)\right]^{+}.
$$
Using \eqref{eq:4.3}, similarly to the proof of Lemma \ref{le:4.1}, we have for all $n \in \mathbb{N}$,
$$
\begin{aligned}
& w_n \in E,\\
& u(1) \leq |u|_{2}\leq\|u\| \leq 1,\\
& u_n^2 \leq w_n^2+2.
\end{aligned}
$$
Setting $v_n:=G\left(u_n\right)$, we thus deduce that for $n \in \mathbb{N}$,
$$v_n=G\left(\sqrt{u_n^2}\right) \leq G\left(\sqrt{w_n^2+2}\right) \leq \tilde{G}\left(\sqrt{w_n^2+2}\right) \leq c e^{2\pi} \widetilde{G}\left(w_n\right).  $$
For given $\varepsilon>0$ and $n$  sufficiently large, we have
$$
v_n=G\left(u_n\right)=G\left((1+|u|_2^2)^{-\frac{1}{2}} w_n+u_n(1)\right) \leq G\left(w_n+\varepsilon\right) \quad \text { in } I .
$$
We now fix $\varepsilon>0$ and define $F_{\varepsilon}: \mathbb{R} \rightarrow \mathbb{R}$ by $F_{\varepsilon}(t)=\min \left\{G(|t|+\varepsilon), c e^{2\pi} \widetilde{G}(t)\right\}$. Then $F_{\varepsilon}$ has at most $\gamma$-critical growth and satisfies (A). Moreover, we have $v_n \leq F_{\varepsilon}\left(w_n\right)$ and therefore
$$
\begin{aligned}
\Psi_2\left(u_n\right) &\leq \int_0^1 v_n(r) \log \frac{1}{r} \int_0^r v_n(\rho) d \rho d r\\ & \leq \int_0^1  F_{\varepsilon}\left(w(r)\right) \log \frac{1}{r} \int_0^r F_{\varepsilon}\left(w(\rho)\right) d \rho d r \\
& =\int_{I}\int_{I}\log\frac{1}{|x-y|}~F_{\varepsilon}(w(x))F_{\varepsilon}(w(y))dxdy
\end{aligned}
$$
for $n$ sufficiently large. Define $\Phi_{\varepsilon}: E \rightarrow \mathbb{R}$,
$$
\begin{aligned}
\Phi_{\varepsilon}(w)=&\int_{I}\int_{I}\log\frac{1}{|x-y|}~F_{\varepsilon}(w(x))F_{\varepsilon}(w(y))dxdy\\
=&\int_{\mathbb{R}}\int_{\mathbb{R}}\log\frac{1}{|x-y|}~1_{I}F_{\varepsilon}(w(x))1_{I}F_{\varepsilon}(w(y))dxdy.
\end{aligned}
$$
Then Proposition \ref{pro:3.4} implies that
$$
\Psi_2\left(u_n\right) \leq \Phi_{\varepsilon}\left(w\right) \rightarrow \Phi_{\varepsilon}(0) \quad \text { as } n \rightarrow \infty
$$
since $w \rightharpoonup 0$ in $H_0^\frac{1}{2}\left(I\right)$. Moreover, since $G(0)=0$, for $\varepsilon>0$ sufficiently small, we have
$$
\Phi_{\varepsilon}(0)=\int_{I}\int_{I}\log\frac{1}{|x-y|}~F_{\varepsilon}(0)F_{\varepsilon}(0)dxdy=\int_{I}\int_{I}\log\frac{1}{|x-y|}~G(\varepsilon)G(\varepsilon)dxdy,
$$
whereas
$$
\int_{I}\int_{I}\log\frac{1}{|x-y|}~G(\varepsilon)G(\varepsilon)dxdy \rightarrow \int_{I}\int_{I}\log\frac{1}{|x-y|}~G(0)G(0)dxdy=0 \quad \text { as } \varepsilon \rightarrow 0^{+}.
$$
It thus follows that
$$
\limsup\limits_{n \rightarrow \infty} \Psi_2\left(u_n\right) \leq 0=\Psi_2(0)=\Psi_2(u)
$$
and therefore (\ref{eq:4.11}) holds in this case.

Case (ii). With $v_n=G\left(u_n\right)$ and $v=G(u)$, we then deduce, as in the proof of Proposition \ref{pro:3.5}, that $v_n$ is bounded in $L^{s_0}\left(\mathbb{R}\right)$ with $s_0=1+\frac{t}{ \pi}>1$, and that
$$
v_n \rightarrow v \quad \text { in } L^s\left(\mathbb{R}\right) \quad \text { for } 1 \leq s<s_0.
$$
Moreover, since
$$
\Psi_2\left(u_n\right)=\int_0^1 v_n(r) \log \frac{1}{r} \int_0^r v_n(r) d \rho d r,
$$
 therefore
$$
\begin{aligned}
& \Psi_2\left(u_n\right)-\Psi_2(u)\\
=&\int_0^1 v_n(r) \log \frac{1}{r} \int_0^r v_n(r) d \rho d r-\int_0^1 v(r) \log \frac{1}{r} \int_0^r  v(\rho) d \rho d r \\
=&\int_0^1 v_n(r) \log \frac{1}{r} \int_0^r \left[v_n(\rho)-v(\rho)\right] d \rho d r+\int_0^1 \left[v_n(r)-v(r)\right] \log \frac{1}{r} \int_0^r v(r) d \rho d r.
\end{aligned}
$$
We may argue exactly as in the proof of Proposition \ref{pro:3.5} to see that $\Psi_2\left(u_n\right)-\Psi_2(u) \rightarrow 0$ as $n \rightarrow \infty$. Hence (\ref{eq:4.11}) also holds in this case.

Finally, by Fatou's lemma, we also have $\Psi_1(u) \leq \liminf _{n \rightarrow \infty} \Psi_1\left(u_n\right)$. Combining this with (\ref{eq:4.11}), we conclude that $\Psi(u) \geq \limsup _{n \rightarrow \infty} \Psi\left(u_n\right)$. Hence $u$ is a maximizer of $\Psi$ in $E_{\infty}$.
\end{proof}

\begin{proposition}\label{pro:4.4}
    Suppose that $G$ satisfies $(A)$ and has at least $\gamma$-critical growth for some $\gamma<1$. Then there exists a sequence of functions $u_n \in E_{\infty} \cap L^{\infty}\left(\mathbb{R}\right)$ with $\Psi\left(u_n\right) \rightarrow \infty$ as $n \rightarrow \infty$.
\end{proposition}
\begin{proof}
    It suffices to take the sequence $\{u_n\}$ considered in the proof of Proposition \ref{pro:3.3}, since $\Psi\left(u_n\right)=\Phi\left(u_n\right)$ for every $n \in \mathbb{N}$ and $\left\|u_n\right\| \leq 1$ for $n$ sufficiently large.
\end{proof}

\section{Symmetry of the positive solutions }
In this section we prove Theorem \ref{th:1.6},  by the moving plane method. Let $G$ and $g:=G^{\prime}$ satisfy the assumptions of Theorem \ref{th:1.6}, and $(u, w)$ be a classical solution of the nonlinear system
\begin{equation}\label{eq:6.2}
    \left\{\begin{aligned}
(- \Delta)^{\frac{1}{2}} u+u & =\theta w g(u) & & \text { in } \mathbb{R} \\
(- \Delta)^{\frac{1}{2}}  w & = \pi G(u) & & \text { in } \mathbb{R}
\end{aligned}\right.
\end{equation}
satisfying
\begin{equation}\label{eq:6.3}
  u>0~\forall x\in\mathbb{R},~ u \in L^{\infty}\left(\mathbb{R}\right) \quad \text { and } \quad w(x) \rightarrow-\infty ~ \text { as }~|x| \rightarrow \infty,
\end{equation}
where $\theta> 0$ is a fixed Lagrangian multiplier.

Before using moving plane method, we need to prove the decay of the classical solution $u$ at infinity.

\subsection{The decay of the  solution  }

\begin{lemma}\label{le:6.1}
    Suppose that $u$ satisfies (\ref{eq:6.2}) and (\ref{eq:6.3}), then
    \begin{equation}\label{eq:6.4}
        u(x) \rightarrow 0 \quad \text { as }|x| \rightarrow \infty.
    \end{equation}
\end{lemma}
\begin{proof}
We prove it by contradiction, otherwise, there exist $\lambda_{\infty}:=\limsup _{|x| \rightarrow \infty} u(x)>0$, and we let $\{x_n\}$ be a sequence in $\mathbb{R}$ with $\left|x_n\right| \rightarrow \infty$ and $u\left(x_n\right) \rightarrow \lambda_{\infty}$ as $n \rightarrow \infty$. Then there exists $n_0 >0$ with
\begin{equation}\label{eq:6.5}
    \lambda_n:=\sup _{x \in [\mathbb{R}\setminus I\left(x_n\right)]} u(x)<\frac{7}{6} \lambda_{\infty}<\frac{6}{5} u\left(x_n\right) \quad \text { for } n \geq n_0,
\end{equation}
where
$$I(x_n
):=\{x\in\mathbb{R}: |x-x_n|<1\}.$$
By (\ref{eq:6.3}), we have
\begin{equation}\label{eq:6.5.1}
    (- \Delta)^{\frac{1}{2}} u+u=\theta w g(u) \leq 0
\end{equation}
in $I\left(x_n\right)$ for $n \geq n_0$, with $n_0$ sufficiently large if necessary.

Let $\gamma(x)=\frac{x^2}{1 +  x^2}+1$, we firstly verify $ (- \Delta)^{\frac{1}{2}}\gamma+\gamma \geq0$.
We calculate according to the definition of the fractional Laplace operator:
$$
(-\Delta)^{1/2} \gamma(x) = C_{1,\frac{1}{2}} \text{P.V.} \int_{-\infty}^{\infty} \frac{\gamma(x) - \gamma(y)}{|x - y|^2} dy,
$$
where \( C_{1,\frac{1}{2}}= \frac{1}{\pi} \) is the normalization constant. We see
\begin{equation}\label{eq:f1}
    \begin{aligned}
(-\Delta)^{1/2} \gamma(x) &= \frac{1}{\pi} \text{P.V.} \int_{-\infty}^{\infty}\frac{\frac{x^2 - y^2}{ (1 +  x^2)(1 +  y^2)}}{(x - y)^2}dy \\
&= \frac{1}{\pi} \text{P.V.} \int_{-\infty}^{\infty} \frac{x + y}{(x - y) (1 +  x^2)(1 +  y^2)} dy \\
&= \frac{1}{\pi(1 +  x^2)} \text{P.V.} \int_{-\infty}^{\infty} \frac{x + y}{(x - y) (1 +  y^2)} dy.
\end{aligned}
\end{equation}
Denote
$$
J(x) = \text{P.V.} \int_{-\infty}^{\infty} \frac{x + y}{(x - y) (1 +  y^2)} dy,
$$
we have
$$
\begin{aligned}
J(x) &= \text{P.V.} \int_{-\infty}^{\infty} \left[ \frac{2x}{(x - y)(1 +  y^2)} - \frac{1}{1 +  y^2} \right] dy \\
&= 2x \cdot \text{P.V.} \int_{-\infty}^{\infty} \frac{1}{(x - y)(1 +  y^2)} dy - \int_{-\infty}^{\infty} \frac{1}{1 +  y^2} dy\\
&=2x\cdot J_1(x)-\pi,
\end{aligned}
$$
and
$$
\begin{aligned}
J_1(x)&= \text{P.V.} \int_{-\infty}^{\infty} \frac{1}{(x - y)(1 +  y^2)} dy\\
& =\text{P.V.} \int_{-\infty}^{\infty} \frac{1}{1 +  x^2} \left[ \frac{1}{x - y} +  \frac{x + y}{1 +  y^2} \right] dy\\
&=\frac{1}{1 +  x^2}\text{P.V.}\int_{-\infty}^{\infty}  \frac{1}{x - y} dy+\frac{x}{1 +  x^2}\int_{-\infty}^{\infty}  \frac{1}{1 +  y^2}  dy+\frac{1}{1 +  x^2}\int_{-\infty}^{\infty}  \frac{y}{1 +  y^2}  dy\\
&=0+\frac{x}{1 +  x^2}\cdot\pi+0=\frac{\pi x}{1 +  x^2}.
\end{aligned}
$$
Thus we have
\begin{align*}
J(x)
=\frac{2\pi x^2}{1 +  x^2} -\pi.
\end{align*}
By \eqref{eq:f1}, we deduce that
\begin{align*}
(-\Delta)^{1/2} \gamma(x)
&=\frac{1}{\pi(1 +  x^2)} \cdot J(x)=\frac{2 x^2}{(1 +  x^2)^2} -\frac{1}{1 +  x^2}=\frac{ x^{2} - 1}{  (1 +  x^{2})^{2}}.
\end{align*}
Thus for $x\in\mathbb{R}$, we derive
$$
(- \Delta)^{\frac{1}{2}}\gamma+\gamma =\frac{ x^{2} - 1}{  (1 +  x^{2})^{2}}+\frac{x^2}{1 +  x^2}+1=\frac{ 2x^{2} (x^2+2)}{  (1 +  x^{2})^{2}}\geq0.
$$

Let's now consider the function
$$\gamma_n(x)=\frac{5}{6}\lambda_n  \gamma\left(x-x_n\right).$$
We see
\begin{equation}\label{eq:  6.5.2}
    (- \Delta)^{\frac{1}{2}}\gamma_n+\gamma_n =\frac{5}{6}\lambda_n\left((- \Delta)^{\frac{1}{2}}\gamma\left(x-x_n\right)+\gamma\left(x-x_n\right)\right)\geq0\quad\text{in }I(x_n).
\end{equation}
Take the difference between (\ref{eq:6.5.1}) and (\ref{eq:  6.5.2}), we have
\begin{equation}\label{20250619-e1}
 (- \Delta)^{\frac{1}{2}}\left(\gamma_n-u\right)+\left(\gamma_n-u\right) \geq 0 \quad \text { in } I\left(x_n\right) \text { for } n \geq n_0.
\end{equation}
Moreover, notice that $$\gamma(x)=\frac{x^2}{1 +  x^2}+1=2-\frac{1}{1+x^2}\geq \frac{3}{2},\quad\text{for }\forall |x|\geq1,$$
we get
\begin{equation}\label{eq:  6.5.25}
   \gamma_n(x) \geq\frac{5}{4} \lambda_n\geq u(x)\quad\text{in }\mathbb{R}\setminus I(x_n).
\end{equation}

Combining \eqref{20250619-e1} with \eqref{eq:  6.5.25}, we can use the maximum principle for the fractional Laplacian to obtain that $$\gamma_n-u \geq 0\quad\text{in}~I\left(x_n\right)~\text { for } n \geq n_0.$$
Indeed, if the function $v:=\gamma_n-u$ attain negative minimum in $I\left(x_n\right)$ for $n \geq n_0$, then by the continuity of $v$ there exists a $x^0\in I\left(x_n\right)$ such that
$$v(x^0)=\min_{I\left(x_n\right)}v<0.$$
And by \eqref{eq:  6.5.25}, we see that $v(y)\geq0$ for any $y\in\mathbb{R}\setminus I(x_n)$.
Therefore,
\begin{equation}\label{20250619-e2}
 (- \Delta)^{\frac{1}{2}}v(x^0)=\frac{1}{\pi} \text{P.V. }\left(\int_{I(x_n)}+\int_{\mathbb{R}\setminus I(x_n)}\right)\frac{v(x^0)-v(y)}{|x^0-y|^{2}}dy<0,~~ \text { for } n \geq n_0.
\end{equation}
On the other hand, from \eqref{20250619-e1}, we have
$$ (- \Delta)^{\frac{1}{2}}v(x^0)\geq -v(x^0)>0,$$
which contradicts with \eqref{20250619-e2}.

Therefore, $\gamma_n-u \geq 0$ in $I\left(x_n\right)$, and then $u\left(x_n\right) \leq \gamma_n\left(x_n\right)=\frac{5\lambda_n}{6} $ for $n \geq n_0$. Thus we have $\lambda_n\geq \frac{6}{5} u\left(x_n\right)$  for $n \geq n_0,$ that is contradict with (\ref{eq:6.5}). Hence, (\ref{eq:6.4}) holds.
\end{proof}

For $N\geq2$, P. Felmer, A. Quaas and J. Tan \cite{FQT} used the properties of  Bessel kernel to prove that:
there is a continuous functions $w$ in $\mathbb{R}^N$ satisfying
$$(-\Delta)^s w(x) + w(x)=0\quad \text{if }~|x| > 1$$
in the classical sense, and
$w(x) \leq c\frac{1}{|x|^{N+2s}} .$ Then by comparison arguments, the authors proved that the positive solution for the fraction equation
$$(-\Delta)^su+u=f(x,u),\quad\text{in}~\mathbb{R}^N;~\lim\limits_{|x|\to\infty}u(x)=0$$
satisfies the decay
$u=O(\frac{1}{|x|^{N+2s}}).$
This approach fails for one-dimension case.
However, we prove
\begin{lemma}\label{le:6.2}
     If u satisfies (\ref{eq:6.2}) and (\ref{eq:6.3}), then $u(x)=O\left(\frac{1}{|x|^{2}}\right)$ as $|x| \rightarrow \infty$.
\end{lemma}

\begin{proof} First,
by (\ref{eq:6.3}) we have
\begin{equation}\label{20250627-e5}
(- \Delta)^{\frac{1}{2}} u+u=\theta w g(u) \leq 0\quad\text{in} ~ \mathbb{R} \backslash \overline{I_d}
\end{equation}
for $d$ sufficiently large, where
$$I_d:=\{x\in\mathbb{R}: |x|<d\}.$$

Now, consider the function
$$
x \mapsto \mu(x)=\|u\|_{L^{\infty}\left(\mathbb{R}\right)} \frac{1}{1+|x|^{2}}:=\|u\|_{L^{\infty}\left(\mathbb{R}\right)} f(x).
$$
We  verify that $\mu$ satisfies $(- \Delta)^{\frac{1}{2}} \mu+\mu \geq 0$ in $\mathbb{R} \backslash \overline{I_d}$. Indeed,
$$\mathscr{F}(f)(\xi)  = \int_{-\infty}^{\infty} \frac{1}{1 + x^2} e^{-i2\pi x \xi}dx=\pi e^{-2\pi |\xi|},$$
therefore
$$\mathscr{F}\left[ (-\Delta)^{1/2} f \right](\xi) = |\xi| \mathscr{F}(f)(\xi)  = \pi |\xi| e^{-2\pi |\xi|},$$
and
$$(-\Delta)^{1/2} f =\mathscr{F}^{-1}\left[|\xi| \mathscr{F}(f)\right]= \int_{-\infty}^{\infty} \pi |\xi| e^{-2\pi |\xi|} e^{i2\pi x \xi}  d\xi=\frac{1 - x^2}{(1 + x^2)^2}.$$
Consequently, we have the estimate
$$
(- \Delta)^{\frac{1}{2}} \mu+\mu=\|u\|_{L^{\infty}\left(\mathbb{R}\right)}\left[\frac{1 - x^2}{(1 + x^2)^2}+\frac{1}{1+|x|^{2}}\right]=\|u\|_{L^{\infty}\left(\mathbb{R}\right)}\cdot\frac{2}{(1 + x^2)^2}\geq0.
$$
This together with \eqref{20250627-e5} implies that the function $\mu-u$ satisfies
\begin{equation}\label{20250627-e1}
(- \Delta)^{\frac{1}{2}} (\mu-u)+(\mu-u) \geq 0\quad\text{in} ~ \mathbb{R} \backslash \overline{I_d}.
\end{equation}
Moreover,  by the definition of $\mu$ that
\begin{equation}\label{20250627-e2}
\mu-u \geq 0 \quad\text{in} ~ \overline{I_d}
\end{equation}
and by Lemma \ref{le:6.1} that
\begin{equation}\label{20250627-e3}
\mu-u(x) \rightarrow 0 \quad\text{as} ~  |x| \rightarrow \infty.
\end{equation}
Combining \eqref{20250627-e1}-\eqref{20250627-e3}, we can use the maximum principle for the fractional Laplacian to obtain that $$\mu-u \geq 0\quad\text{in}~\mathbb{R}\setminus \overline{I_d}~\text { for } n \geq n_0.$$
Indeed, if the function $\mu-u$ attain negative minimum in $\mathbb{R}\setminus \overline{I_d}$ for $n \geq n_0$, then by \eqref{20250627-e3} and the continuity of $\mu-u$ there exists a $x^0\in \overline{I_d}$ such that
$$(\mu-u)(x^0)=\min_{\overline{I_d}}v<0.$$
And by \eqref{20250627-e2}, we see that $(\mu-u)(y)\geq0$ for any $y\in \overline{I_d}$.
Therefore,
\begin{equation}\label{20250627-e4}
 (- \Delta)^{\frac{1}{2}}(\mu-u)(x^0)=\frac{1}{\pi} \text{P.V. }\left(\int_{\overline{I_d}}+\int_{\mathbb{R}\setminus \overline{I_d}}\right)\frac{(\mu-u)(x^0)-(\mu-u)(y)}{|x^0-y|^{2}}dy<0,~~ \text { for } n \geq n_0.
\end{equation}
On the other hand, from \eqref{20250627-e1}, we have
$$ (- \Delta)^{\frac{1}{2}}(\mu-u)(x^0)\geq -(\mu-u)(x^0)>0,$$
which contradicts with \eqref{20250627-e4}.

Thus the function $\mu-u$ can not attain negative minimum in  $\mathbb{R} \backslash \overline{I_d}$. Therefore $u \leq \mu$ in $\mathbb{R} \backslash \overline{I_d}$, which implies that $u(x)=O\left(\frac{1}{|x|^{2}}\right)$ as $|x| \rightarrow \infty$.
\end{proof}

\begin{lemma}\label{le:6.25}
     If $(u,w)$ satisfies (\ref{eq:6.2}) and (\ref{eq:6.3}), then
\begin{equation}\label{1409-e12}
w(x)=O\left(\log|x|\right),\quad \text{as}~ |x| \rightarrow \infty.
     \end{equation}
\end{lemma}
\begin{proof}
  By Lemma \ref{le:6.2}, we have $u=O(\frac{1}{|x|^2})$.
  By our assumption, $g(0)=G(0)=0$ and  $g$ is Lipschitz continuous on $[0,\infty)$, we have $g(u)\leq C|u|$ and $G(u)\leq C|u|^2.$ Thus  \begin{equation}\label{1515}
    g(u)=O(\frac{1}{|x|^2})~ \text{and}~G(u)=O(\frac{1}{|x|^4}),\quad\text{as}~|x|\to\infty.
  \end{equation}
  By \eqref{1515}, it is easy to see that $G(u)\in L^1_{\log}(\mathbb{R})$.

Now, we estimate $w(x)$ as
  $$
  \begin{aligned}
  w= &\log\frac{1}{|x|} \ast G(u)=\int_{\mathbb{R}}\log\frac{1}{|x-y|}G(u(y))dy
  \\
  =&-\log|x|\int_{\mathbb{R}}G(u(y))dy-\int_{\mathbb{R}}G(u(y))\log\left|1-\frac{y}{x}\right|dy\\
  :=&-C\log|x|-J(x).
  \end{aligned}
  $$
 Fix $|x|>>1$. We divide $\mathbb{R}$ into three parts
 $$\begin{aligned}
I_1:=&\left\{|y|\leq\frac{1}{2}|x|\right\},~~I_2:=\left\{|y|>\frac{1}{2}|x|\right\}\cap\left\{|y-x|\geq\frac{|x|}{2}\right\},\\
 I_3:=&\left\{|y|>\frac{1}{2}|x|\right\}\cap\left\{|y-x|<\frac{|x|}{2}\right\},\end{aligned}$$
 and then
 $$J(x)=\left(\int_{I_1}+\int_{I_2}+\int_{I_3}\right)G(u(y))\log\left|1-\frac{y}{x}\right|dy:=J_1+J_2+J_3.$$
 (i) For $y\in I_1$, we see $\left|\log|1-\frac{y}{x}|\right|\leq C\frac{|y|}{|x|}
 $. Then
 $$|J_1(x)|\leq \frac{C}{|x|}\int_{I_1}G(u(y))|y|dy\leq\frac{C}{|x|}.$$
 (ii) For $y\in I_2$, we have $\left|\log|1-\frac{y}{x}|\right|<\log3+\log(1+|y|).
 $ Thus
 $$|J_2(x)|\leq \log3\int_{I_2}G(u(y))dy+\int_{I_2}G(u(y))\log(1+|y|)dy\leq C.$$
(iii) For $y\in I_3$, we infer $\frac{1}{2}|x|<|y|<\frac{3}{2}|x|$ and $\left|\log|1-\frac{y}{x}|\right| \leq 2 \frac{|y-x|}{|x|}$.
 Then by \eqref{1515},
  $$
  \begin{aligned}
   |J_3(x)|\leq &\frac{2}{|x|}\int_{I_3}|G(u(y))||y-x|dy\\
   \leq &\frac{C}{|x|^5}\int_{I_3}|y-x|dy\\
   \leq &\frac{C}{|x|^5}\int_0^{\frac{|x|}{2}}zdz=\frac{C}{|x|^3}.
  \end{aligned}
  $$

Combining (i)-(iii), we get $|J(x)|\leq C$ for $|x|\to\infty.$
Therefore, $w=-C\log|x|-J(x)=O(\log|x|)$ as $|x|\to\infty$.

\end{proof}

\begin{lemma}\label{le:6.256}
     If $(u,w)$ satisfies (\ref{eq:6.2}) and (\ref{eq:6.3}), then
\begin{equation}\label{1409-e1}
     [u]_{H^{\frac{1}{2}}(\mathbb{R})}<\infty.
     \end{equation}
\end{lemma}
\begin{proof}
Let $f(x)=\theta w g(u)$, then by Lemma \ref{le:6.2} and Lemma \ref{le:6.25}, we have $u=O(\frac{1}{|x|^2})$ and $f(x)=O(\frac{\log|x|}{|x|^2})$ as $|x|\to\infty$.    Thus,
  $$(-\Delta)^{\frac{1}{2}}u u=(f-u)u\in L^1(\mathbb{R}).$$
Then we use
  \begin{align*}
  \int_{\mathbb{R}}(- \Delta)^{\frac{1}{2}} u udx
=&C_{1,\frac{1}{2}}\text{P.V.}\int_{\mathbb{R}}\int_{\mathbb{R}}\frac{u(x)^2-u(y)u(x)}{|x-y|^{2}}dy dx  \\
=&C_{1,\frac{1}{2}}\text{P.V.}\int_{\mathbb{R}}\int_{\mathbb{R}}\frac{(u(y))^2-u(x)u(y)}{|x-y|^{2}}dx dy\\
=&\frac{1}{2}C_{1,\frac{1}{2}}\text{P.V.}\int_{\mathbb{R}}\int_{\mathbb{R}}\frac{(u(x)-u(y))^2}{|x-y|^{2}}dy dx\\
=&\frac{1}{2}C_{1,\frac{1}{2}}[u]^2_{H^{\frac{1}{2}}(\mathbb{R})}
  \end{align*}
 to infer \eqref{1409-e1}.
\end{proof}

\subsection{Symmetry of the  solution }

By Lemma \ref{le:6.2} and Lemma \ref{le:6.256}, it is easy to verify that $u\in H^{\frac{1}{2}}(\mathbb{R}).$ Here, $u$ is the classic positive solution of (\ref{eq:6.2})-(\ref{eq:6.3}).
We now proceed to the proof of Theorem \ref{th:1.6} using the moving plane method, note
$$
\begin{aligned}
\Sigma_\lambda&:=\left\{x \in \mathbb{R}: x_1<\lambda\right\}, \\
T_\lambda=&\partial \Sigma_\lambda=\left\{x \in \mathbb{R}:x_1=\lambda\right\},\\
x^\lambda&=2\lambda-x,\quad\text{for }\lambda \in \mathbb{R} ,
\end{aligned}
$$
and
$$
u^\lambda(x)=u\left(x^\lambda\right), \quad w^\lambda(x)=w\left(x^\lambda\right) \quad \text { for } x \in \mathbb{R}, \lambda \in \mathbb{R} ,
$$
$$
u_\lambda=u^\lambda-u, \quad w_\lambda=w^\lambda-w.
$$
Obviously, the difference functions $u_\lambda$ and $w_\lambda$ satisfy the system
\begin{equation}\label{eq:6.7}
    \left\{\begin{aligned}
(- \Delta)^{\frac{1}{2}} u_\lambda+u_\lambda & =\theta\left[w_\lambda g\left(u^\lambda\right)+h_\lambda(x) w u_\lambda\right] \\
(- \Delta)^{\frac{1}{2}} w_\lambda & = K_\lambda(x) u_\lambda
\end{aligned} \quad \text { in } \Sigma_\lambda,\right.
\end{equation}
where
$$
h_\lambda(x):= \begin{cases}\frac{g\left(u^\lambda(x)\right)-g(u(x))}{u_\lambda(x)}, & \text { for } u_\lambda(x) \neq 0, \\ 0, & \text { for } u_\lambda(x)=0,\end{cases}
$$
and
$$
K_\lambda(x):= \begin{cases}\frac{G\left(u^\lambda(x)\right)-G(u(x))}{u_\lambda(x)}, & \text { for } u_\lambda(x) \neq 0 ,\\ 0, & \text { for } u_\lambda(x)=0.\end{cases}
$$
\begin{remark}\label{re:6.8}
    (i) Recall that
    $$  w(x)=\int_{\mathbb{R}} \log \frac{1}{|x-y|}  G(u(y)) d y,$$
    by the anti-symmetry of $u_\lambda$ and $w_\lambda$, we have
    $$  w_\lambda(x)=\int_{\mathbb{R}} \log \frac{1}{|x-y|} K_\lambda(y) u_\lambda(y) d y=\int_{\Sigma_\lambda} \log \frac{\left|x-y^\lambda\right|}{|x-y|} K_\lambda(y) u_\lambda(y) d y.$$
    Moreover,
    $  w_\lambda(x)\geq 0 \text { in } \Sigma_\lambda$ if  $u_\lambda \geq 0 \text { in } \Sigma_\lambda$. Indeed, since $G$ is a non-decreasing function, we have $G(u^\lambda)\geq G(u)$ from $u_\lambda=u^\lambda-u \geq 0$, and then $K_\lambda\geq0$. Therefore, we deduce that  $ w_\lambda \geq 0 \text { in } \Sigma_\lambda$ by $\log \frac{\left|x-y^\lambda\right|}{|x-y|}>0$ for every $x, y \in \Sigma_\lambda$.

In addition, by the expression of $w_\lambda$, if $u_\lambda \equiv0$, we have $w_\lambda \equiv 0$ in $\Sigma_\lambda$, if $u_\lambda>0$, then $w_\lambda>0 $ in $\Sigma_\lambda$.

     (ii) By our assumption, $g$ is Lipschitz continuous, then there exists a constant $C=C(u)>0$ such that $    \left\|h_\lambda\right\|_{L^{\infty}\left(\Sigma_\lambda\right)} \leq C ,$ for $  \lambda \in \mathbb{R}.$
\end{remark}

Define $v^{+(-)}:=\max(\min) \{v, 0\}$ the non-negative (non-positive) part of $v$.
For every $\lambda \in \mathbb{R}$, we have the following estimate.
 \begin{lemma}\label{le:6.3}
 There exists a constant $\mu>0$ such that
$$
\left\|w_\lambda^{-}\right\|_{L^2\left(\Sigma_\lambda\right)} \leq \mu c_\lambda\left\|u_\lambda^{-}\right\|_{L^2\left(\Sigma_\lambda\right)} \quad \text { for all } \lambda \in \mathbb{R},
$$
where
$$
c_\lambda:=\left(\int_{\Sigma_\lambda^-}\left|\lambda-y_1\right|\left[g\left(\xi_\lambda(y)\right)\right]^2d y\right)^{\frac{1}{2}} \quad \text { and } \quad \Sigma_\lambda^-:=\left\{x \in \Sigma_\lambda: u_\lambda(x)<0\right\},
$$
where $\xi_\lambda(x)=u(x)+\gamma_\lambda(x) u^\lambda(x)$, $\gamma_\lambda(x) \in( 0,1)$.
\end{lemma}
\begin{proof}
First, by the elementary inequality
$$
\log{(1+|x|)}\leq\sqrt{|x|}\quad \forall~ x\in\mathbb{R}
$$
and Remark \ref{re:6.8} (i), for $x,y\in \Sigma_\lambda$ we have the estimate\\
(pay attention to $w_\lambda^-\leq0$, $u_\lambda^-\leq0$, $\log\frac{|x-y^\lambda|}{|x-y|}>0,~K_\lambda(y)\geq0$ for $x,y\in\Sigma_\lambda$)
\begin{equation}\label{eq:6.8}
    \begin{aligned}
    0\leq -w_\lambda^{-}(x) &=-\int_{\Sigma_\lambda^-} \log \frac{\left|x-y^\lambda\right|}{|x-y|} K_\lambda(y) u_\lambda^-(y) d y\\
&\leq-\int_{\Sigma_\lambda^-} \log  \left(1+\frac{\left|y-y^\lambda\right|}{|x-y|}\right) K_\lambda(y) u_\lambda^-(y) d y\\
&\leq-\int_{\Sigma_\lambda^-} \left(\frac{\left|y-y^\lambda\right|}{|x-y|} \right)^\frac{1}{2}K_\lambda(y) u_\lambda^-(y) d y\\
&=-\int_{\Sigma_\lambda^-} \left(\frac{2\left(\lambda-y_1\right)}{|x-y|} \right)^\frac{1}{2}K_\lambda(y) u_\lambda^{-}(y) d y\\
&=-\sqrt{2} \int_{\Sigma_\lambda^-} \frac{\left(\lambda-y_1\right)^\frac{1}{2}}{|x-y|^\frac{1}{2}} g\left(\xi_\lambda(y)\right) u_\lambda^{-}(y) d y,
\end{aligned}
\end{equation}
where $\xi_\lambda(x)=u(x)+\gamma_\lambda(x) u^\lambda(x)$ with $\gamma_\lambda(x) \in( 0,1).$
Take the $L^2$ norm on $\Sigma_\lambda$ of both sides of inequality (\ref{eq:6.8}), we have
\begin{equation}\label{eq:6.9}
\begin{aligned}
\| w_\lambda^{-}\|_{L^2(\Sigma_\lambda)}&\leq\sqrt{2}\left\| \int_{\Sigma_\lambda^-} \frac{\left(\lambda-y_1\right)^\frac{1}{2}}{|\cdot-y|^\frac{1}{2}} g\left(\xi_\lambda(y)\right) u_\lambda^{-}(y) d y\right\|_{L^2(\Sigma_\lambda)}.
\end{aligned}
\end{equation}
By the Hardy-Littlewood-Sobolev inequality \cite[Theorem 4.3]{Lieb}
$$\left\|\frac{1}{|\cdot|^\alpha}\ast f\right\|_{L^r(\mathbb{R}^N)}\leq C\|f\|_{L^p(\mathbb{R}^N)},\quad \frac{1}{p}+\frac{\alpha}{N}=1+\frac{1}{r}$$
and the Cauchy-Schwarz inequality, we have
\begin{equation}\label{eq:6.10}
\begin{aligned}
       & \left\| \int_{\Sigma_\lambda^-} \frac{\left(\lambda-y_1\right)^\frac{1}{2}}{|\cdot-y|^\frac{1}{2}} g\left(\xi_\lambda(y)\right) u_\lambda^{-}(y) d y\right\|_{L^2(\Sigma_\lambda)}\\
   \leq&\mu'\int_{\Sigma_\lambda^{-}} \left| \left(\lambda-y_1\right)^\frac{1}{2} g\left(\xi_\lambda(y)\right) u_\lambda^{-}(y) \right|dy\\
    \leq &\mu'\left(\int_{\Sigma_\lambda^{-}} \left|\lambda-y_1\right|g\left(\xi_\lambda(y)\right)^2dy\right)^{\frac{1}{2}}\cdot\|u^-\|_{L^2(\Sigma_\lambda)}\\
    =&\mu' c_\lambda\|u^-\|_{L^2(\Sigma_\lambda)},
\end{aligned}
\end{equation}
with the constant $\mu'>0$ independent of $\lambda$.
Combining (\ref{eq:6.8}), (\ref{eq:6.9}), (\ref{eq:6.10}), we claim that
$$
\left\|w_\lambda^{-}\right\|_{L^2\left(\Sigma_\lambda\right)} \leq \mu c_\lambda\left\|u_\lambda^{-}\right\|_{L^2\left(\Sigma_\lambda\right)} \quad \text { for all } \lambda \in \mathbb{R},
$$
with $\mu=\sqrt{2}\mu'>0$ independent of $\lambda$.
\end{proof}

\begin{remark}\label{re:6.3}
Fix $\lambda\in\mathbb{R}$. By Lemma \ref{le:6.2}, we have $\xi_\lambda(x)=O(\frac{1}{x^2})$ as $|x|\to\infty$. By our assumption, $g(0)=0$ and $g$ is Lipschitz continuous, thus $g(\xi_\lambda)\leq C|\xi_\lambda|$. Thus $c_\lambda$ obtained in Lemma \ref{le:6.3} is well defined, i.e., $c_\lambda(x)<\infty$ for any $x\in\mathbb{R}$.

Let $x\in\Sigma_\lambda$. By Lemma \ref{le:6.1}, we have $\xi_\lambda(x)\rightarrow0$ as $|\lambda|\rightarrow\infty$. Combining this with the assumption that $g(0)=0$, we deduce that $c_\lambda \rightarrow 0$ as $|\lambda| \rightarrow \infty.$
\end{remark}

 \textbf{For the regular Laplacian, we see that
 $$\|u^-\|^2_{H^{1}_0(\Omega)}=\int_{\Omega}(- \Delta u)  u^-dx+\int_{\Omega}u  u^-dx.$$
 The equality cannot be generalized to fractional Laplacian cases due to the non-locality of the fractional Laplacian. However,}
 \begin{lemma} \label{le:5.8.1}For any $u\in H^s(\mathbb{R}^N)$,
\begin{equation}\label{0702-e1}
   \|u^-\|^2_{H^{s}({\mathbb{R}^N)}}\leq2C_{N,s}^{-1}\int_{\mathbb{R}^N}(- \Delta)^{s} u u^-dx+\int_{\mathbb{R}^N}uu^-dx.
    \end{equation}
 \end{lemma}
 \begin{proof}
 By $u=u^++u^-$, $\int_{\mathbb{R}^N} \frac{u^+(x)u^-(x)}{|x-y|^{N+2s}}dx=0$ and $$-u(y)u^-(x)=-u^+(y)u^-(x)-u^-(y)u^-(x)\geq -u^-(y)u^-(x),$$ we have
\begin{align*}
  \int_{\mathbb{R}^N}(- \Delta)^{s} u u^-dx
=&C_{N,s}P.V.\int_{\mathbb{R}^N}\int_{\mathbb{R}^N}\frac{u(x)u^-(x)-u(y)u^-(x)}{|x-y|^{N+2s}}dy dx  \\
  \geq &C_{N,s}P.V.\int_{\mathbb{R}^N}\int_{\mathbb{R}^N}\frac{(u^-(x))^2-u^-(y)u^-(x)}{|x-y|^{N+2s}}dy dx\\
=&C_{N,s}P.V.\int_{\mathbb{R}^N}\int_{\mathbb{R}^N}\frac{(u^-(y))^2-u^-(x)u^-(y)}{|x-y|^{N+2s}}dx dy\\
=&\frac{1}{2}C_{N,s}P.V.\int_{\mathbb{R}^N}\int_{\mathbb{R}^N}\frac{(u^-(x)-u^-(y))^2}{|x-y|^{N+2s}}dy dx\\
=&\frac{1}{2}C_{N,s}[u^-]^2_{H^{s}(\mathbb{R}^N)}.
  \end{align*}
Hence, \eqref{0702-e1} holds.
 \end{proof}
 Using the anti-symmetry of $u_\lambda$, we have
\begin{lemma}\label{le:5.8.2}
    For $\Sigma_\lambda=\left\{x \in \mathbb{R}: x_1<\lambda\right\}$,
    $$
    \int_{\mathbb{R}}(- \Delta)^{\frac{1}{2}} u_\lambda(x) u^- _\lambda(x)+ u_\lambda u^- _\lambda dx=2\int_{\Sigma_\lambda}(- \Delta)^{\frac{1}{2}} u_\lambda(x) u^- _\lambda(x)+ u_\lambda u^- _\lambda dx.
    $$
\end{lemma}
\begin{proof}
    For $x\in\Sigma_\lambda$, $u_\lambda(x)=u(x^\lambda)-u(x)$, there exist
    $$
    \begin{aligned}
    &u_\lambda(x^\lambda)=u(x)-u(x^\lambda)=-u_\lambda(x),\\
    &u_\lambda^{-}(x^\lambda)=(u(x)-u(x^\lambda))^{-}=-(u(x^\lambda)-u(x))^{-}=-u_\lambda^{-}(x).
    \end{aligned}
    $$
    Then
    $$
   \begin{aligned}
    (- \Delta)^{\frac{1}{2}} u_\lambda(x^\lambda) =&C_{1,\frac{1}{2}}\text{P.V.}\int_{\mathbb{R}}\frac{u_\lambda(x^\lambda)-u_\lambda(y)}{\left|x^\lambda-y\right|^2}dy\\
=&C_{1,\frac{1}{2}}\text{P.V.}\int_{\mathbb{R}}\frac{u_\lambda(x^\lambda)-u_\lambda(y^\lambda)}{\left|x^\lambda-y^\lambda\right|^2}dy\\
=&-C_{1,\frac{1}{2}}\text{P.V.}\int_{\mathbb{R}}\frac{u_\lambda(x)-u_\lambda(y)}{\left|x-y\right|^2}dy\\
=&- (- \Delta)^{\frac{1}{2}} u_\lambda(x),
   \end{aligned}
    $$
    since $\left|x^\lambda-y^\lambda\right|=\left|x-y\right|$.
    Therefore,
    $$
    \begin{aligned}
        &\int_{\mathbb{R}}(- \Delta)^{\frac{1}{2}} u_\lambda(x)u_\lambda^{-}(x)+u_\lambda u_\lambda^{-}dx\\
        =& \int_{\Sigma_\lambda}(- \Delta)^{\frac{1}{2}} u_\lambda(x)u_\lambda^{-}(x)+u_\lambda u_\lambda^{-}dx
        +\int_{\mathbb{R}\setminus\Sigma_\lambda}(- \Delta)^{\frac{1}{2}} u_\lambda(x)u_\lambda^{-}(x)+u_\lambda u_\lambda^{-}dx\\
        = &\int_{\Sigma_\lambda}(- \Delta)^{\frac{1}{2}} u_\lambda(x)u_\lambda^{-}(x)+u_\lambda u_\lambda^{-}dx
        +\int_{\Sigma_\lambda}(- \Delta)^{\frac{1}{2}} u_\lambda(x^\lambda)u_\lambda^{-}(x^\lambda)+u_\lambda(x^\lambda) u_\lambda^{-}(x^\lambda)dx\\
        =&2 \int_{\Sigma_\lambda}(- \Delta)^{\frac{1}{2}} u_\lambda(x)u_\lambda^{-}(x)+u_\lambda u_\lambda^{-}dx.
    \end{aligned}
    $$
\end{proof}
\begin{lemma}\label{le:6.4}
    There exists $\bar{\lambda}<0$ such that $u_\lambda \geq 0$ in $\Sigma_\lambda$ for $\lambda \leq \bar{\lambda}$.
\end{lemma}
\begin{proof}
  By (\ref{eq:6.3}) and Remark \ref{re:6.8}-(ii), we can choose $\lambda_1<0$ such that $w \leq 0$ in $\Sigma_\lambda$ for $\lambda \leq \lambda_1$. Multiplying the first equation in (\ref{eq:6.7}) by $u_\lambda^{-}$, using Lemma \ref{le:6.3}, Lemma \ref{le:5.8.1}, and Lemma \ref{le:5.8.2}, we have the estimate
$$
\begin{aligned}
\left\|u_\lambda^{-}\right\|_{H^\frac{1}{2}\left(\Sigma_\lambda\right)}^2
&\leq\int_{\mathbb{R}}2C_{1,\frac{1}{2}}^{-1}(- \Delta)^{\frac{1}{2}} u_\lambda u_\lambda^{-}+u_\lambda u_\lambda^{-}dx\\
&=2\int_{\Sigma_\lambda}2C_{1,\frac{1}{2}}^{-1}(- \Delta)^{\frac{1}{2}} u_\lambda u_\lambda^{-}+u_\lambda u_\lambda^{-}dx\\
&\leq C\int_{\Sigma_\lambda}(- \Delta)^{\frac{1}{2}} u_\lambda u_\lambda^{-}+u_\lambda u_\lambda^{-}dx\\
&=C\int_{\Sigma_\lambda} \theta\left[w_\lambda g\left(u^\lambda\right) u_\lambda^{-}+h_\lambda(x) w\left(u_\lambda^{-}\right)^2\right] d x \\
& \leq C \int_{\Sigma_\lambda} \theta w_\lambda g\left(u^\lambda\right) u_\lambda^{-} d x\quad(\text{since}~w\leq0~\text{in}~\Sigma_\lambda) \\
&\leq C\theta\left\|w_\lambda^{-}\right\|_{L^2\left(\Sigma_\lambda\right)}\left\|g\left(u^\lambda\right)\right\|_{L^{\infty}\left(\Sigma_\lambda\right)}\left\|u_\lambda^{-}\right\|_{L^2\left(\Sigma_\lambda\right)}\\
&\leq C\theta\mu  c_\lambda\|g(u)\|_{L^{\infty}\left(\mathbb{R}\right)}\left\|u_\lambda^{-}\right\|_{L^2\left(\Sigma_\lambda\right)}^2.
\end{aligned}
$$
By Remark \ref{re:6.3}, there exists $\bar{\lambda} \leq
\lambda_1$ such that for all $\lambda\leq\bar{\lambda}$,
$$
C\theta\mu c_\lambda\|g(u)\|_{L^{\infty}\left(\mathbb{R}\right)}<1.
$$
Therefore,
$$
\left\|u_\lambda^{-}\right\|_{H^\frac{1}{2}\left(\Sigma_\lambda\right)}^2=0,
$$
and thus $u_\lambda^{-} \equiv 0$ on $\Sigma_\lambda$ for $\lambda \leq \bar{\lambda}$, witch implies that $u_\lambda=u^\lambda-u\geq0$ for $\lambda \leq \bar{\lambda}$.
\end{proof}
\begin{remark}\label{re:6.4}
In the same way, we can deduce that there exists $\bar{\lambda}>0$ such that $u_\lambda \leq 0$ in $\Sigma_\lambda$ for $\lambda \geq \bar{\lambda}$.
\end{remark}

Lemma \ref{le:6.4} provides a starting point for moving plane. Indeed,
let $$\lambda_1:=\sup \left\{\lambda \in \mathbb{R}: u_\mu \geq 0 \text { in } \Sigma_\mu, \lambda\geq\mu\right\},$$ then by Lemma \ref{le:6.4}, we have $\lambda_1>-\infty$. And by $u>0$ and Lemma \ref{le:6.1} $\lim\limits_{|x|\to\infty}u(x)=0$, we know $\lambda_1<+\infty$.


\begin{lemma}\label{le:6.6} Fixed a $\lambda\in\mathbb{R}$.
If $u_\lambda>0$ in $\Sigma_\lambda$, then there exists $\epsilon>0$ such that $u_{\lambda_\epsilon} \geq 0$ in $\Sigma_{\lambda_\epsilon}$ for ${\lambda_\epsilon} \in$ $(\lambda, \lambda+\epsilon)$.
\end{lemma}
\begin{proof}
For $d>0$, by (\ref{eq:6.3}) and Remark \ref{re:6.8}-(i), we have
$$
w \leq 0 \quad \text { in } \Sigma_{\lambda_\epsilon} \backslash I_d \text { for every } {\lambda_\epsilon} \in \mathbb{R}, ~d \text{  large enough if necessary.}
$$
We notice that by Remark \ref{re:6.8}, there exist $\epsilon>0$ such that $u_{\lambda_\epsilon} > 0$ in $\Sigma_{\lambda_\epsilon}\cap I_d$ for ${\lambda_\epsilon} \in$ $(\lambda, \lambda+\epsilon)$, $\lambda \in \mathbb{R}$.
As the proof of Lemma \ref{le:6.4}, we use Lemma \ref{le:6.3}, Lemma \ref{le:5.8.1}, and Lemma \ref{le:5.8.2}, multiply the first equation in (\ref{eq:6.7}) with $u_{\lambda_\epsilon}^{-}$and integrate, replace $\lambda$ with $\lambda_\epsilon $, then we have the estimate
$$
\begin{aligned}
\left\|u_{\lambda_\epsilon}^{-}\right\|_{H^\frac{1}{2}\left(\Sigma_{\lambda_\epsilon}\right)}^2
&\leq C\int_{\mathbb{R}}(- \Delta)^{\frac{1}{2}} u_{\lambda_\epsilon}u_{\lambda_\epsilon}^-+u_{\lambda_\epsilon} u_{\lambda_\epsilon}^{-}dx\\
&=2C\int_{\Sigma_{\lambda_\epsilon}}((- \Delta)^{\frac{1}{2}} u_{\lambda_\epsilon}+u_{\lambda_\epsilon} ))u_{\lambda_\epsilon}^{-}dx\\
&\leq2C\int_{\Sigma_{\lambda_\epsilon} \backslash I_d} \theta\left[w_{\lambda_\epsilon}g\left(u^{\lambda_\epsilon}\right) u_{\lambda_\epsilon}^{-}+h_{\lambda_\epsilon}(x) w\left(u_{\lambda_\epsilon}^{-}\right)^2\right] d x \\
& \leq 2C\int_{\Sigma_{\lambda_\epsilon} \backslash I_d} \theta w_{\lambda_\epsilon}g\left(u^{\lambda_\epsilon}\right) u_{\lambda_\epsilon}^{-} d x\quad (\text{since}~w\leq0~\text{in}~\Sigma_{\lambda_\epsilon}\backslash I_d) \\
& \leq 2C\theta\left\|w_{\lambda_\epsilon}^{-}\right\|_{L^2\left(\Sigma_{\lambda_\epsilon}\right)}\left\|g\left(u^{\lambda_\epsilon}\right)\right\|_{L^{\infty}\left(\Sigma_{\lambda_\epsilon}\right)}\left\|u_{\lambda_\epsilon}^{-}\right\|_{L^2\left(\Sigma_{\lambda_\epsilon}\right)} \\
&\leq 2C\theta\mu c_{\lambda_\epsilon}\|g(u)\|_{L^{\infty}\left(\mathbb{R}\right)}\left\|u_{\lambda_\epsilon}^{-}\right\|_{L^2\left(\Sigma_{\lambda_\epsilon}\right)}^2.
\end{aligned}
$$
Here, $$
c_{\lambda_\epsilon}=\left(\int_{\Sigma_{\lambda_\epsilon}^-}\left|{\lambda_\epsilon}-y_1\right|\left[g\left(\xi_{\lambda_\epsilon}(y)\right)\right]^2d y\right)^{\frac{1}{2}} \leq\left(\int_{\mathbb{R}\setminus I_d}\left|{\lambda_\epsilon}-y_1\right|\left[g\left(\xi_{\lambda_\epsilon}(y)\right)\right]^2d y\right)^{\frac{1}{2}}
$$
for ${\lambda_\epsilon} \in(\lambda, \lambda+\epsilon)$. Then for enough large $d$, we have $$
2C\theta\mu c_{\lambda_\epsilon}\|g(u)\|_{L^{\infty}\left(\mathbb{R}\right)}<1,
$$
and thus $u_{\lambda_\epsilon}^{-} \equiv 0$ on $\Sigma_{\lambda_\epsilon}$ for ${\lambda_\epsilon} \in$ $({\lambda}, {\lambda}+\epsilon)$, which implies that $u_{\lambda_\epsilon}=u^{\lambda_\epsilon}-u\geq0$ for ${\lambda_\epsilon} \in$ $({\lambda}, {\lambda}+\varepsilon)$.
\end{proof}

Now, we claim $u_{\lambda_1}\equiv0$.  If not, by the maximum principle of fractional Laplacian, we have $u_{\lambda_1}>0$. Then by Lemma \ref{le:6.6}, there exist $\epsilon>0$ such that $u_{\lambda_\epsilon} \geq 0$ in $\Sigma_{\lambda_\epsilon}$ for ${\lambda_\epsilon} \in$ $(\lambda_1, \lambda_1+\epsilon)$, which contradicts the definition of $\lambda_1$. Therefore, $u_{\lambda_1}\equiv0$, along with $w_{\lambda_1} \equiv 0$ by Remark \ref{re:6.8}.
Therefore, up to a translation, we see that $u$ is even and radially decreasing. So far, the proof of Theorem \ref{th:1.6}  has been completed. \quad

\subsection*{Funding}
\begin{sloppypar}
This study was partially supported by NSFC (No. 11901532).

\noindent{\bf Data availability statement:} Our manuscript has no data associated or further material.

 \noindent{\bf Ethical Approval:}  All data generated or analyzed during this study are included in this article

 \noindent{\bf Conflict of interest:} The authors declare no conflict of interest.

\end{sloppypar}

\end{document}